\documentclass{amsart}
\usepackage{amsmath,amsfonts,amsthm}
\usepackage{amssymb,latexsym}
\usepackage{graphics}
\usepackage{rotating}
\usepackage[colorlinks]{hyperref}

\usepackage[all]{xypic}
\usepackage{amssymb}
\usepackage{xspace, enumerate}

\theoremstyle{plain}
\theoremstyle{definition}
\newtheorem{theorem}{Theorem}[section]
\newtheorem{lemma}[theorem]{Lemma}

\newtheorem{construction}[theorem]{Construction}
\newtheorem{corollary}[theorem]{Corollary}

\newtheorem{definition}[theorem]{Definition}
\newtheorem{example}[theorem]{Example}
\newtheorem{counterexample}[theorem]{Counterexample}
\newtheorem{problem}[theorem]{Problem}

\newtheorem{note}[theorem]{Note}
\newtheorem{jadval}[theorem]{Table}
\newtheorem{convention}[theorem]{Convention}
\theoremstyle{remark}

\newtheorem*{remark}{Remark}
\numberwithin{equation}{section}
\newcommand{\SP}{\: \: \: \: \:}
\title{On special subgroups of fundamental group}
\author[F. Ayatollah Zadeh Shirazi, F. Ebrahimifar, M. A. Mahmoodi]{Fatemah Ayatollah Zadeh Shirazi, \\ Fatemeh Ebrahimifar, Mohammad Ali Mahmoodi}
\begin{document}
\begin{abstract}
Suppose $\alpha$ is a nonzero cardinal number,
$\mathcal I$ is an ideal on
arc connected topological space $X$, and
${\mathfrak P}_{\mathcal I}^\alpha(X)$ is the subgroup of $\pi_1(X)$
(the first fundamental group of $X$) generated by homotopy classes of
$\alpha\frac{\mathcal I}{}$loops.
The main aim of this text is to study ${\mathfrak P}_{\mathcal I}^\alpha(X)$s
and compare them.
Most interest is in $\alpha\in\{\omega,c\}$ and $\mathcal
I\in\{\mathcal P_{fin}(X),\{\varnothing\}\}$, where $\mathcal
P_{fin}(X)$ denotes the collection of all finite subsets of $X$.
We denote  ${\mathfrak P}_{\{\varnothing\}}^\alpha(X)$ with
 ${\mathfrak P}^\alpha(X)$. We
prove the following statements:
\\
$\bullet$ for arc connected topological spaces $X$ and $Y$
    if
    ${\mathfrak P}^\alpha(X)$ is isomorphic to ${\mathfrak P}^\alpha(Y)$
    for all infinite cardinal number $\alpha$, then
    $\pi_1(X)$ is isomorphic to $\pi_1(Y)$;
\\
$\bullet$ there are arc connected topological spaces $X$ and $Y$
    such that $\pi_1(X)$ is isomorphic to $\pi_1(Y)$ but
    ${\mathfrak P}^\omega(X)$ is not isomorphic to ${\mathfrak P}^\omega(Y)$;
\\
$\bullet$ for arc connected topological space $X$ we have
    ${\mathfrak P}^\omega(X)\subseteq{\mathfrak P}^c(X)
    \subseteq\pi_1(X)$;
\\
$\bullet$ for Hawaiian earring $\mathcal X$, the sets
    ${\mathfrak P}^\omega({\mathcal X})$, ${\mathfrak P}^c({\mathcal X})$,
    and $\pi_1({\mathcal X})$
    are pairwise distinct.
\\
So  ${\mathfrak P}^\alpha(X)$s and  ${\mathfrak P}_{\mathcal I}^\alpha(X)$s
will help us to classify the class of all arc connected topological spaces with
isomorphic fundamental groups.
\end{abstract}
\maketitle
\noindent {\small {\bf 2010 Mathematics Subject Classification:}  55Q05 \\
{\bf Keywords:}} $\alpha-$arc, $\alpha\frac{\mathcal I}{}$arc,
$\alpha\frac{\mathcal I}{}$loop, fundamental group, Hawaiian earring.
\section{Introduction}
\noindent The main aim of algebraic topology is ``classifying the
topological spaces''.  One of the first concepts introduced in
algebraic topology is ``fundamental group''. As it has been
mentioned in \cite[page1]{M84}, fundamental groups are introduced
by Poincar$\acute{\rm e}$. In this text we consider special
subgroups of fundamental group. Explicitly we pay attention to
path homotopy classes induced by loops which are ``enough one to one''.
We have the following sections:
{\small \begin{itemize}
\item[1.] Introduction
\item[2.] What is an $\alpha\frac{\mathcal I}{}$arc?
\item[3.] New subgroups
\item[4.] A useful remark
\item[5.] Primary properties of ${\mathfrak P}_{\mathcal I}^\alpha(X)$s
\item[6.] Some preliminaries on Hawaiian earring
\item[7.] $\mathfrak P^c(\mathcal X)$ is a proper subset of $\pi_1(\mathcal X)$
\item[8.] $\mathfrak P_{{\mathcal P}_{fin}(\mathcal Y)}^c(\mathcal Y)$
    is a proper subset of $\pi_1(\mathcal Y)$
\item[9.] Main examples and counterexamples
\item[10.] Main Table
\item[11.] Two spaces having fundamental groups isomorphic to
    Hawaiian earring's fundamental group
\item[12.] A distinguished counterexample
\item[13.] A diagram and a hint
\item[14.] A strategy for future and conjecture
\item[15.] Conclusion
\end{itemize}}
\noindent Our main conventions located in section 2, although
there are conventions in other sections too. Briefly, we
introduce our new subgroups in Section 3 and obtain their primary
properties in Section 5. Sections 6, 7 and 8 contain basic lemmas
for our counterexamples in Section 9. Regarding these three
sections 7, 8, and 9 we see $\mathfrak P^\omega (\mathcal
X)\subset\mathfrak P^c (\mathcal X)\subset\pi_1(\mathcal X)$ where
$\mathcal X$ is Infinite or Hawaiian earring and ``$\subset$''
means strict inclusion; also we see $\mathfrak P_{{\mathcal
P}_{fin}(\mathcal Y)}^\omega (\mathcal Y) \subset\mathfrak
P_{{\mathcal P}_{fin}(\mathcal Y)}^c (\mathcal
Y)\subset\pi_1(\mathcal Y)$ ($\mathcal Y$ is introduced in
Section 2). However Counterexamples of Section 9 are essential
for Main Table in Section 10, which shows probable inclusion
relations between different $\mathfrak P_{\mathcal I}^\alpha(X)$
for a fix $X$ (arc connected locally compact Hausdorff
topological space), $\alpha\in\{\omega,c\}$ and ${\mathcal
I}\in\{\{\varnothing\},{\mathcal P}_{fin}(X), {\mathcal P}(X)\}$
where ${\mathcal P}(X)$ is the power set of $X$ and $\mathfrak
P_{{\mathcal P}(X)}^\alpha(X)$ is just $\pi_1(X)$ (the
fundamental group of $X$) by Section 5. We continue to discover
the properties of ``our new subgroups'' in Sections 12 and 13, as
a matter of fact in Sections 11 and 12 we see $\pi_1(\mathcal
X)\cong\pi_1(\mathcal W)$ and $\mathfrak P^\omega(\mathcal
X)\ncong\mathfrak P^\omega(\mathcal W)$ ($\mathcal W$ is
introduced in Section 2), consequently we have a diagram and two
problems in Section 13. As a matter of fact using the diagram of
Section 13 and ``Distinguished Example'' in Section 12, we try to
show ``these new subgroups'' can make meaningful subclasses of a
\textit{class of arc connected locally compact Hausdorff
topological spaces with the isomorphic fundamental groups}.
\\
Remembering all the conventions during reading the text is highly recommended.
\begin{convention}
A topological space $X$ is an arc connected space, if for all
$a,b\in X$ with $a\neq b$ there exists a continuous one to one
map $f:[0,1]\to X$ with $f(0)=a$ and $f(1)=b$. In this text all
spaces assumed to be Hausdorff, locally compact, and arc
connected with at least two elements.
\end{convention}
\begin{remark}
Let $X$ be an arbitrary set. We call $\mathcal I\subseteq\mathcal
P(X)$, an ideal on $X$, if:
\begin{itemize}
\item $\mathcal I\neq\varnothing$,
\item If $A,B\in\mathcal I$, then $A\cup B\in\mathcal I$,
\item If $B\subseteq A$ and $A\in\mathcal I$, then $B\in\mathcal I$.
\end{itemize}
The collection of all finite subsets of $X$,
${\mathcal P}_{fin}(X)$, is one of the most famous
ideals on $X$.
\end{remark}
\noindent In this text ZFC+GCH (we recall that GCH or
\textit{Generalized Continuum Hypothesis} indicates that for
transfinite cardinal number $\beta$, there is not any cardinal
number $\gamma$ with $\beta<\gamma<2^\beta$, i.e.
$2^\beta=\beta^+$ \cite{Ho99}) is assumed and by ``$\subset$'' we
mean strict inclusion. Whenever $G$ is a group isomorphic to
group $H$, we write $G\cong H$. Also $G \ncong H$ means that $G$
is not isomorphic to $H$. Whenever $g\in G$ and $A\subseteq G$,
then $<A>$ denotes the subgroup of $G$ generated by $A$, denote
$<\{g\}>$ simply by $<g>$. We recall that $\omega$ is the
cardinality of ${\mathbb N}$ (the set of all natural numbers
$\{1,2,\ldots\}$) and $c$ is the cardinality of ${\mathbb R}$
(the set of all real numbers). We denote the cardinality of $A$
by $|A|$. For cardinal numbers (real numbers) $\alpha,\beta$ we
denote the maximum of $\{\alpha,\beta\}$ by $\max(\alpha,\beta)$
also we denote the minimum of $\{\alpha,\beta\}$ by
$\min(\alpha,\beta)$.
\\
In addition for $n\in\mathbb N$, consider ${\mathbb R}^n$ under Euclidean norm.
Also consider ${\mathbb
S}^1=\{(x,y)\in{\mathbb R}^2:x^2+y^2=1\}$ as a subspace
of ${\mathbb R}^2$ (or $\{e^{i\theta}:\theta\in[0,2\pi]\}$ as a
subspace of $\mathbb C$, the set of all complex numbers).
\section{What is an $\alpha\frac{\mathcal I}{}$arc?}
\noindent The concept of $\alpha\frac{\mathcal I}{}$arc is a generalization of
$\alpha-$arc which is
originated from \cite{A07} and then in \cite{AH16}.  However a $1-$arc or briefly arc is a one to one map
$f:[0,1]\to X$.
\begin{definition}
For nonzero cardinal number $\alpha$, and ideal ${\mathcal I}$ on
$X$, the continuous map $f:Y\to X$ is called an
$\alpha\frac{\mathcal I}{}$map if there exists $A\in\mathcal I$
such that for all $x\in X\setminus A$, $|f^{-1}(x)|<\alpha+1$ .
In particular for infinite cardinal number $\alpha$, the
continuous map $f:Y\to X$ is an $\alpha\frac{\mathcal I}{}$map if
there exists $A\in\mathcal I$ such that for all $x\in X\setminus
A$, $|f^{-1}(x)|<\alpha$.
\\
We call $\alpha\frac{\mathcal I}{}$map $f:[0,1]\to X$,
$\alpha\frac{\mathcal I}{}$arc. We call $\alpha\frac{\mathcal
I}{}$map $f:[0,1]\to X$ with $f(0)=f(1)=a$, an
$\alpha\frac{\mathcal I}{}$loop with base point $a$.
\\
We use briefly terms $\alpha-$map,
$\alpha-$arc, and $\alpha-$loop respectively instead of
$\alpha\frac{\{\varnothing\}}{}$map,  $\alpha\frac{\{\varnothing\}}{}$arc,
and $\alpha\frac{\{\varnothing\}}{}$loop.
\end{definition}
\noindent We want to study subgroups of $\pi_1(X)$ generated by
path homotopy equivalence classes of $\alpha-$loops and
$\alpha\frac{\mathcal I}{}$loops for nonzero cardinal number
$\alpha$ and ideal $\mathcal I$ on $X$. We pay special attention
to  $\alpha\frac{\mathcal I}{}$loops for
$\alpha\in\{\omega,c\}$ and ${\mathcal I}\in\{{\mathcal
P}_{fin}(X),\{\varnothing\}\}$. We use the following spaces and loops in
most counterexamples in this text.
\begin{convention}\label{convention-asli}
Suppose $p\in\mathbb N$, let
\begin{eqnarray*}
    \mathcal X & := & \left\{\dfrac1ne^{2\pi i\theta}+\dfrac{i}{n}:n\in{\mathbb N},\theta\in[0,1]\right\} \\
    & (= & {\displaystyle\bigcup_{n\in\mathbb N}
        \left\{(x,y)\in\mathbb R^2:x^2+(y-\dfrac1n)^2=\dfrac1{n^2}\right\}})\SP{\rm (Hawaiian \: earring)}\\
    \mathcal Y & := & \bigcup\left\{\dfrac{1}{2^{n+1}} \mathcal X+\dfrac1n:n\in\mathbb N\right\}\cup[0,1] \\
    \mathcal Z & := & \left\{\dfrac1ke^{2\pi i(x-k-\frac14)}+\dfrac{i}{k}:k\in\{1,...,p\},x\in[0,1]\right\} \\
    \mathcal W & := & \left\{\dfrac{1}{2^{n+1}}e^{2\pi i\theta}+\dfrac{1}{n}+\dfrac{i}{2^{n+1}}
        :n\in\mathbb N,\theta\in[0,1]\right\}\cup[0,1] \\
    C_n & := & \left\{\frac1ne^{2\pi it}+\frac{i}{n}:t\in[0,1]\right\} \\
    & (= & \left\{(x,y)\in\mathbb R^2:x^2+(y-\dfrac1n)^2=\dfrac1{n^2}\right\})
    \\
    & & \SP\SP\SP\SP
    {\rm (circle \: with \: radius}\:\frac1n\:{\rm and \: center}\:\frac{i}{n}(n\in\mathbb N)) \\
    \mathcal V & := & {\displaystyle\bigcup_{n\in\mathbb N}\left\{(x,y,z)\in\mathbb R^3:y^2+(z-\dfrac1n)^2=\dfrac{1}{n^2}
    \wedge0\leq x\leq\dfrac1n\right\}}
\end{eqnarray*}
moreover define $f_{\mathcal X}:[0,1]\to\mathcal X$,
$f_{\mathcal Y}:[0,1]\to\mathcal Y$ and $f_{\mathcal Z}:[0,1]\to\mathcal Z$ with:
\[f_{\mathcal X}(x)=\left\{\begin{array}{lc}
    \dfrac1ne^{2\pi i(n(n+1)x-n-\frac14)}+\dfrac{i}{n} &
        \dfrac1{n+1}\leq x\leq\dfrac1n,n\in{\mathbb N}\: , \\
        & \\
    0 & x=0 \:,
    \end{array}\right.\]
\[f_{\mathcal Y}(x)=\left\{\begin{array}{lc}
    \dfrac{f_{\mathcal X}(4xn(n+1)-(2n+1))}{2^{n+1}}+\dfrac1n &
        \dfrac{2n+1}{4n(n+1)}\leq x\leq\dfrac1{2n},n\in{\mathbb N} \: ,\\
        & \\
    2(n+1)(2n-1)x+(2-2n) & \dfrac1{2(n+1)}\leq x\leq\dfrac{2n+1}{4n(n+1)},n\in{\mathbb N}\:, \\
    & \\
    2-2x & \dfrac12\leq x\leq1 \:,\\
    &\\
    0 & x=0\:,
    \end{array}\right.\]
and
\[f_{\mathcal Z}(x)=\frac1ke^{2\pi i(px-k-\frac14)}+\dfrac{i}{k}\SP(
\dfrac{k-1}{p}\leq x\leq\dfrac{k}{p},k\in\{1,...,p\})\:.\]
$\:$ \\
{\it Note: Consider 0 as the base point of all spaces in this convention}
\vspace{5mm}
\begin{center}
\begin{tabular}{|c|c|}
    \hline & \\
    \includegraphics[scale=0.4]{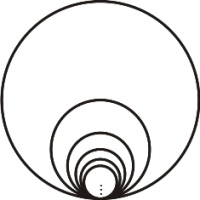} &
        \includegraphics[scale=0.4]{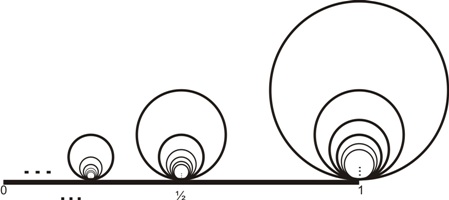} \\
    (Figure of $\mathcal X=f_{\mathcal X}[0,1]$) &
        (Figure of $\mathcal Y=f_{\mathcal Y}[0,1]$)  \\ \hline & \\
    \includegraphics[scale=0.4]{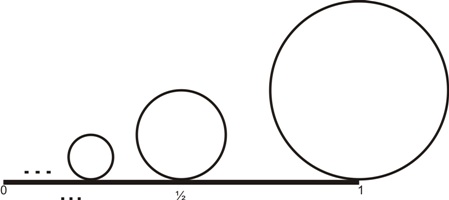} &
    \includegraphics[scale=0.4]{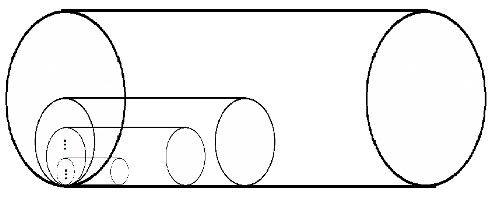} \\
        (Figure of $\mathcal W$) & (Figure of $\mathcal V$) \\ \hline
\end{tabular}
\end{center}
\vspace{5mm}
\end{convention}
\begin{example}\label{taha10}
1) The map  $f_{\mathcal X}:[0,1]\to\mathcal X$ is an $\alpha-$loop if and only if
$\alpha>\omega$, since:
\[|f_{\mathcal X}^{-1}(x)|=\left\{\begin{array}{lc} 1 & x\neq0 \: , \\ \omega & x=0 \: .
    \end{array}\right.\]
In addition for each nonzero cardinal number $\alpha$ and ideal $\mathcal I$ on $\mathcal X$
with $\{0\}\in\mathcal I$,  $f_{\mathcal X}:[0,1]\to\mathcal X$ is an $\alpha\frac{\mathcal I}{}$loop.
\\
2) The map $f_{\mathcal Y}:[0,1]\to\mathcal Y$ is an $\alpha\frac{\mathcal I}{}$loop if and only if
``$\alpha>\omega$'' or ``$\alpha\geq2$ and
$\{\frac1n:n\in{\mathbb N}\}\in\mathcal I$'', since:
\[|f_{\mathcal Y}^{-1}(x)|=\left\{\begin{array}{lc}
    \omega & x\in\{\frac1n:n\in{\mathbb N}\} \: , \\
    2 & {\rm otherwise}\:. \end{array}\right.\]
In particular $f_{\mathcal Y}:[0,1]\to\mathcal Y$ is an $\alpha\frac{{\mathcal P}_{fin}(Y)}{}$loop
if and only if $\alpha\geq c$.
\\
3) The map $f_{\mathcal Z}:[0,1]\to\mathcal Z$ is an $\alpha-$loop if and only if
$\alpha>p$.
In addition for all nonzero cardinal number $\alpha$ and ideal $\mathcal I$ on $\mathcal X$
with $\{0\}\in\mathcal I$,  $f_{\mathcal X}:[0,1]\to\mathcal X$ is an $\alpha\frac{\mathcal I}{}$loop.
\end{example}
\section{New subgroups}
\noindent In this section
we introduce $\mathfrak P^\alpha_\mathcal I(X)$ as a subgroup of
$\pi_1(X)$.
\\
We recall that
for continuous maps $f,g:[0,1]\to X$
    with $f(1)=g(0)$, we have $f*g:[0,1]\to X$ with $f*g(t)=f(2t)$ for
    $t\in[0,\frac12]$ and $f*g(t)=g(2t-1)$ for $t\in[\frac12,1]$. If
$f:[0,1]\to X$ is a continuous map, $[f]$ denotes its path homotopy equivalence
class, where loops $f,g:[0,1]\to X$ with same base point $a$ are path homotopic
(or $[f]=[g]$) if there exists continuous map
$F:[0,1]\times[0,1]\to X$ with $F(s,0)=f(s)$, $F(s,1)=g(s)$ and $F(0,s)=F(1,s)=a$ for
all $s\in[0,1]$.
\\
\textbf{In the rest of this paper simply
we use term ``\textit{homotopy}'' or ``\textit{homotopic}''
respectively instead of ``\textit{path homotopy}'' or ``\textit{path homotopic}''.}
\\
In addition for two loops
 $f,g:[0,1]\to X$ with same base point $a$, we define
 $[f]*[g]$ as $[f*g]$. The class of all homotopy equivalence
 classes of loops with base point $a$ under operation $*$
 is a group which is denoted by $\pi_1(X,a)$. Whenever $X$
 is arc connected and $a,b\in X$ we have $\pi_1(X,a)\cong\pi_1(X,b)$
 so $\pi_1(X,a)$ is denoted simply by $\pi_1(X)$.
\begin{definition}
For nonzero cardinal number $\alpha$ and ideal $\mathcal I$ by
$\mathfrak P^\alpha_\mathcal I(X,a)$ we mean subgroup of
$\pi_1(X,a)$ generated by homotopy classes of
$\alpha\frac{\mathcal I}{}$loops with base point $a$.
\end{definition}
\begin{theorem}\label{Narges1}
For infinite cardinal number $\alpha$ and ideal $\mathcal I$ on
$X$, if $f,g:[0,1]\to X$ are $\alpha\frac{\mathcal I}{}$arcs
    with $f(1)=g(0)$, then $f*g:[0,1]\to X$
    is an $\alpha\frac{\mathcal I}{}$arc.
Moreover
    $\overline f:[0,1]\to X$ with $\overline f(t)=f(1-t)$ is an
    $\alpha\frac{\mathcal I}{}$arc too.
\end{theorem}
\begin{proof}
Use the fact that for all $x\in X$,
$(f*g)^{-1}(x)=(\frac12f^{-1}(x))\cup (\frac12g^{-1}(x)+\frac12)$,
thus $|(f*g)^{-1}(x)|\leq|f^{-1}(x)|+|g^{-1}(x)|$. Also note to
the fact that for all $x\in X$ we have ${\overline f}^{-1}(x)=\{1-t:t\in f^{-1}(x)\}$,
hence $|{\overline f}^{-1}(x)|=|f^{-1}(x)|$.
\end{proof}
\begin{theorem}\label{Narges01}
For infinite cardinal number $\alpha$, $a\in X$ and ideal
$\mathcal I$ on $X$, we have:
\begin{center}
$\mathfrak P^\alpha_\mathcal
I(X,a)=\{[f]:\:f:[0,1]\to X$ is an $\alpha\frac{\mathcal
I}{}$loop with base point $a\}$.
\end{center}
\end{theorem}
\begin{proof}
Choose $b\in X\setminus\{a\}$. There exists a continuous one to
one map $g:[0,1]\to X$ with $g(0)=a$ and $g(1)=b$. Using
Theorem~\ref{Narges1}, $g*\overline g:[0,1]\to X$ is an
$\alpha\frac{\mathcal I}{}$arc. Thus $[g*\overline
g]\in\{[f]:\:f:[0,1]\to X$ is an $\alpha\frac{\mathcal I}{}$loop
with base point $a\}$, and $\{[f]:\:f:[0,1]\to X$ is an
$\alpha\frac{\mathcal I}{}$loop with base point
$a\}\neq\varnothing$. Using Theorem~\ref{Narges1},
$\{[f]:\:f:[0,1]\to X$ is an $\alpha\frac{\mathcal I}{}$loop with
base point $a\}$ is a subgroup of $\pi_1(X,a)$ which completes
the proof.
\end{proof}
\begin{note}\label{Narges2}
Using Theorem~\ref{Narges01} for $a\in X$ and infinite cardinal
number $\alpha$, for the loop $g:[0,1]\to X$ with base point $a$,
$[g]\in\mathfrak P^\alpha_\mathcal I(X,a)$ if and only if there
exists an $\alpha\frac{\mathcal I}{}$loop $f:[0,1]\to X$ with
base point $a$ homotopic to $g:[0,1]\to X$.
\end{note}
\begin{theorem}\label{jaleb}
For all $a,b\in X$, ideal $\mathcal I$ on $X$
and infinite $\alpha$, $\mathfrak P^\alpha_\mathcal I(X,a)$ and
$\mathfrak P^\alpha_\mathcal I(X,b)$ are isomorphic groups.
\end{theorem}
\begin{proof}
For $a\neq b$, suppose $f:[0,1]\to X$ is a continuous one to one
map (1-arc) such that $f(0)=a$ and $f(1)=b$, and $\overline
f:[0,1]\to X$ is $\overline f(t)=f(1-t)$ for all $t\in[0,1]$.
Using Theorem~\ref{Narges1}, $g:[0,1]\to X$
is an $\alpha\frac{\mathcal I}{}$arc  if and only if
$\overline f*g*f:[0,1]\to X$ is an $\alpha\frac{\mathcal I}{}$arc
too, which leads to the desired result (note: $\varphi:
{\mathfrak P}_{\mathcal I}^\alpha(X,a)\to{\mathfrak P}_{\mathcal
I}^\alpha(X,b)$, with $\varphi([g])=[ \overline f*g*f]$ is an
isomorphism).
\end{proof}
\noindent By the following counterexample the infiniteness  of
$\alpha$ in Theorem~\ref{jaleb} is essential.
\begin{counterexample}
Consider $X=\mathbb S^1\cup[1,2]$ as a
subspace of $\mathbb R^2$) ($X$ and {\large {\bf 9}} are
homeomorph). If $a\in \mathbb S^1$ and $b\in(1,2]$, then:
\begin{itemize}
\item[1.] $\mathfrak P^1_{\mathcal P_{fin}(X)}(X,a)=\pi_1(X,a)\cong\mathbb Z$,
\item[2.] $\mathfrak P^1_{\mathcal P_{fin}(X)}(X,b)=\{e\}$ (where $e$ is the identity of $\pi_1(X,b)$).
\end{itemize}
In particular $\mathfrak P^1_{\mathcal P_{fin}(X)}(X,a)$ and
$\mathfrak P^1_{\mathcal P_{fin}(X)}(X,b)$ are nonisomorphic
(although $X$ is linear connected).
\end{counterexample}
\begin{proof}
(1) By definition $\mathfrak P^1_{\mathcal
P_{fin}(X)}(X,a)\subseteq\pi_1(X,a)(=\mathbb Z)$. On the other
hand $f:\mathop{[0,1]\to X}\limits_{\SP\SP\SP t\mapsto e^{2\pi it}}$ is a $1\frac{\mathcal
P_{fin}(X)}{}$arc and $\pi_1(X,a)=<[f]>\subseteq\mathfrak
P^1_{\mathcal P_{fin}(X)}(X,a)$, which completes the proof.
\\
(2) Suppose $f:[0,1]\to X$ with $f(0)=f(1)=b$ is a continuous map.
If $f\neq b$, then there exists $c\in [1,2]\setminus\{b\}$ with
$c=\inf[0,1]$. Let $s:=\min(c,b)$ and $t:=\max(c,b)$. For all
$y\in(s,t)$ we have $|f^{-1}(y)|\geq2$, and $(s,t)\notin{\mathcal
P_{fin}(X)}$ (since $(s,t)$ is infinite). Therefore $f$ is not a
$1\frac{\mathcal P_{fin}(X)}{}$loop, and the constant loop $b$ is
the unique $1\frac{\mathcal P_{fin}(X)}{}$loop with base point
$b$, thus $\mathfrak P^1_{\mathcal P_{fin}(X)}(X,b)=\{[b]\}=\{e\}$
\end{proof}
\begin{definition}\label{salam}
Regarding Theorem~\ref{jaleb} for infinite cardinal number
$\alpha$ and ideal $\mathcal I$ on $X$, we denote $\mathfrak
P^\alpha_\mathcal I(X,a)$ simply by $\mathfrak P^\alpha_\mathcal
I(X)$ (subgroup of $\pi_1(X)$ generated by homotopy classes of
$\alpha\frac{\mathcal I}{}$loops). We denote $\mathfrak
P^\alpha_{\{\varnothing\}}(X)$ by $\mathfrak P^\alpha(X)$
(subgroup of $\pi_1(X)$ generated by homotopy classes of
$\alpha-$loops).
\\
So for infinite cardinal number $\alpha$ we have (use
Note~\ref{Narges2} and above discussion):
\begin{center}
    $\mathfrak P^\alpha_\mathcal I(X)=\{[f]:\:f:[0,1]\to X$ is an
    $\alpha\frac{\mathcal I}{}$loop$\}$,
\end{center}
and
\begin{center}
    $\mathfrak P^\alpha(X)=\{[f]:\:f:[0,1]\to X$ is an
    $\alpha-$loop$\}$.
\end{center}
\end{definition}
\section{A useful remark}
\noindent For the remain of this text we use the following useful convention.
\begin{convention}\label{good10}
Suppose $X$ and $Y$ are closed subspaces of $Z$
such that $X\cap Y=\{t\}$. For $f:[0,1]\to X\cup Y$ define:
\[f^X(x)=\left\{\begin{array}{lc}
    f(x) & f(x)\in X \: , \\
    t & f(x)\in Y \: .
    \end{array}\right.\]
\end{convention}
\begin{remark}\label{useful}
Suppose $X$ and $Y$ are closed (linear connected) subspaces of $Z$
such that $X\cap Y=\{t\}$. For loops $g,h:[0,1]\to X\cup Y$ with base point $t$ we have:
\begin{itemize}
\item[A.] If $g,h:[0,1]\to X\cup Y$ are homotopic loops, then
    $g^X,h^X:[0,1]\to X$ are homotopic loops (therefore
    $g^X,h^X:[0,1]\to X\cup Y$ are homotopic too).
\item[B.] Let $g[0,1]\subseteq X$ and $h[0,1]\subseteq Y$.
    $g,h:[0,1]\to X\cup Y$ are homotopic if and only if they
    are null-homotopic.
\item[C.] Let $g[0,1]\cup h[0,1]\subseteq X$.
    $g,h:[0,1]\to X\cup Y$ are homotopic if and only if
    $g,h:[0,1]\to X$ are homotopic.
\item[D.] $\pi_1(X,t)$ and $\pi_1(Y,t)$ are subgroups of $\pi_1(X\cup Y,t)$
    and $\pi_1(X,t)\cap\pi_1(Y,t)=\{[t]\}$ where $t$
    denotes the constant arc with value $t$ (as a matter of fact
    the map
    $\mathop{\pi_1(X,t)\to\pi_1(X\cup Y,t)}\limits_{[f] \mapsto [f]\SP\SP}$
    is a group embedding).
\end{itemize}
\end{remark}
\begin{proof}
(A) Suppose $g,h:[0,1]\to X\cup Y$ are
homotopic loops, then there exists a continuous map
$F:[0,1]\times[0,1]\to X\cup Y$ such that $F(s,0)=g(s)$,
$F(s,1)=h(s)$ and $F(0,s)=F(1,s)=t$ for all $s\in[0,1]$. Define continuous map $P:X\cup
Y\to X$ with $P(z)=z$ for $z\in X$ and $P(z)=t$ for $z\in Y$. The
map $P\circ F:[0,1]\times[0,1]\to X$ is continuous, moreover
$P\circ F(s,0)=g^X(s)$, $P\circ F(s,1)=h^X(s)$ and $P\circ F(1,s)=P\circ F(0,s)=t$ for all
$s\in[0,1]$, thus $g^X,h^X:[0,1]\to X\cup Y$ are homotopic.
\\
(B) If $g,h:[0,1]\to X\cup Y$ are homotopic, then by (A),
$g^X,h^X:[0,1]\to X\cup Y$ are homotopic. On the other hand $g^X=t$
(constant function $t$) and $h^X=h$, since $g[0,1]\subseteq X$
and $h[0,1]\subseteq Y$. Therefore $h:[0,1]\to X\cup Y$ is null homotopic
which leads to the desired result.
\end{proof}
\section{Primary properties of ${\mathfrak P}_{\mathcal I}^\alpha(X)$s}
\noindent In this section we study primary properties of
$\mathfrak P^\alpha_\mathcal I(X)$s. It is wellknown that
$\Phi:\pi_1(X,x_0)\times\pi_1(Y,y_0)\to\pi_1(X\times Y,(x_0,y_0))$
with $\Phi([f],[g])=[(f,g)]$ is an isomorphism (for example see
\cite[Theorem 60.1]{M07}) where for $f:[0,1]\to X$ and
$g:[0,1]\to Y$ we have $(f,g):[0,1]\to X\times Y$ with
$(f,g)(t)=(f(t),g(t))$ ($t\in[0,1]$). For transfinite cardinal
numbers $\alpha,\beta$, ideal ${\mathcal I}$ on
    $X$ and ideal $\mathcal J$ on $Y$ we prove
$\Phi({\mathfrak P}_{\mathcal I}^\alpha(X,x_0)\times
{\mathfrak P}_{\mathcal J}^\beta(Y,y_0))
    \subseteq{\mathfrak P}_{{\mathcal I}
    \times{\mathcal J}}^{\max(\alpha,\beta)}(X\times
    Y,(x_0,y_0))$, hence
    ${\mathfrak P}_{\mathcal I}^\alpha(X,x_0)\times
    {\mathfrak P}_{\mathcal J}^\beta(Y,y_0)$ is isomorphic to a subgroup of
    ${\mathfrak P}_{{\mathcal I}\times{\mathcal J}}^{\max(\alpha,\beta)}(X\times
    Y,(x_0,y_0))$.
\begin{theorem}\label{Narges3}
For topological spaces $X$ and $Y$ we have (we recall that $X$ and $Y$
are arc connected locally compact Hausdorff topological spaces with at least two
elements, moreover consider $x_0\in X$, and
$y_0\in Y$):
\\
1. For all $\alpha>c$, nonzero $\beta$ and ideal $\mathcal I$ on $X$ we have
    $\mathfrak P^\alpha_{\mathcal I}(X)=\pi_1(X)=\mathfrak P^\beta_{{\mathcal P}(X)}(X)$.
\\
2. For nonzero cardinal numbers $\alpha,\beta$, $x_0\in X$, and
    ideals $\mathcal I,\mathcal J$ on $X$ we have:
    \begin{itemize}
    \item If $\alpha\leq\beta$, then
            $\mathfrak P_\mathcal I^\alpha(X,x_0)\subseteq\mathfrak P_\mathcal I^\beta(X,x_0)$.
    \item If ${\mathcal I}\subseteq{\mathcal J}$, then
            $\mathfrak P_{\mathcal I}^\alpha(X,x_0)\subseteq\mathfrak P_{\mathcal J}^\alpha(X,x_0)$.
    \item[] Therefore for infinite $\alpha$ we have (base point is $x_0$, whenever it is necessary):
    \item If $\alpha\leq\beta$, then
        $\mathfrak P_\mathcal I^\alpha(X)\subseteq\mathfrak P_\mathcal I^\beta(X)$.
    \item If ${\mathcal I}\subseteq{\mathcal J}$, then
            $\mathfrak P_{\mathcal I}^\alpha(X)\subseteq\mathfrak P_{\mathcal J}^\alpha(X)$;
    \item $\mathfrak P_{\mathcal I\cap \mathcal J}^\alpha(X)
            \subseteq \mathfrak P_{\mathcal I}^\alpha(X)
            \cap\mathfrak P_{\mathcal   J}^\alpha(X)$.
        \end{itemize}
3. For infinite cardinal numbers $\alpha$, $\beta$ and ideals $\mathcal I$ on $X$
    and $\mathcal J$ on $Y$ we have
    \[\Phi({\mathfrak P}_{\mathcal I}^\alpha(X,x_0)\times{\mathfrak P}_{\mathcal J}^\beta(Y,y_0))
    \subseteq{\mathfrak P}_{{\mathcal I}\times{\mathcal J}}^{\max(\alpha,\beta)}(X\times
    Y,(x_0,y_0)),\]
    where ${\mathcal I}\times{\mathcal J}$ is ideal on $X\times Y$ generated by $\{A\times
    B:A\in{\mathcal I},B\in{\mathcal J}\}$ and $\Phi([f],[g])=[(f,g)]$ for
    loops $f:[0,1]\to X$ and $g:[0,1]\to Y$. Hence ${\mathfrak P}_{\mathcal I}^\alpha(X)\times{\mathfrak P}_{\mathcal J}^\beta(Y)$
    is isomorphic to a subgroup of ${\mathfrak P}_{{\mathcal I}\times{\mathcal J}}^{\max(\alpha,\beta)}(X\times
    Y)$.
\\
4. For infinite cardinal numbers $\alpha$, $\beta$, ideal $\mathcal I$
    on $X$, and isomorphism $\Phi:\pi_1(X)\times\pi_1(Y)\to\pi_1(X\times Y)$
    with $\Phi([f],[g])=[(f,g)]$, we have:
\begin{itemize}
    \item[a.] $\Phi({\mathfrak P}_{\mathcal I}^\alpha(X,x_0)\times\pi_1(Y,y_0))\subseteq
        {\mathfrak P}_{{\mathcal I}\times{\mathcal P}(Y)}^\alpha(X\times Y,(x_0,y_0)),$
    \item[b.] $\Phi({\mathfrak P}_{\mathcal I}^\alpha(X,x_0)\times{\mathfrak P}^\beta(Y,y_0))
        \subseteq{\mathfrak P}^\beta(X\times Y,(x_0,y_0))$;
    \item[c.] $\Phi({\mathfrak P}^\alpha(X,x_0)\times{\mathfrak P}^\beta(Y,y_0))
        \subseteq{\mathfrak P}^{\min(\alpha,\beta)}(X\times Y,(x_0,y_0))$.
    \end{itemize}
5. For infinite cardinal numbers $\alpha$, $\beta$, ideal $\mathcal I$
    on $X$, ideal $\mathcal J$ on $Y$, $\mathcal K:=\{A\cup B:A\in\mathcal I,B\in\mathcal J\}$,
    if $X\cap Y=\{t\}$ and $X,Y$ are (linear connected) closed subspaces of $Z$, then we have
    (note that $\mathcal K$ is an ideal on
    $X\cup Y$) (see Convention~\ref{good10} (D)):
    \begin{itemize}
    \item[a.] ${\mathfrak P}_{\mathcal I}^\alpha(X,t){\mathfrak P}_{\mathcal J}^\beta(Y,t)\subseteq
        {\mathfrak P}_{\mathcal K}^{\max(\alpha,\beta)}(X\cup Y,t)$,
    \item[b.] ${\mathfrak P}^\alpha(X,t){\mathfrak P}^\beta(Y,t)\subseteq
        {\mathfrak P}^{\max(\alpha,\beta)}(X\cup Y,t)$.
    \end{itemize}
\end{theorem}
\begin{proof}
(1) and (2) are clear by definition.
\\
(3) If $f:[0,1]\to X$ is an $\alpha\frac{\mathcal I}{}$arc with
base point $x_0$ and $g:[0,1]\to Y$ is a
    $\beta\frac{\mathcal J}{}$arc with base point $y_0$,
    then there exist $A\in{\mathcal I}$
    and $B\in{\mathcal J}$
    such that for all $x\in X\setminus A$ and
    $y\in Y\setminus B$ we have $|f^{-1}(x)|<\alpha$ and
    $|g^{-1}(y)|<\beta$. For $h=(f,g):[0,1]\to X\times Y$
    with $h(t)=(f(t),g(t))$ and
    $(z,w)\in(X\times Y)\setminus(A\times B)$ we have:
    {\small
    \begin{eqnarray*}
    (z,w)\in(X\times Y)\setminus(A\times B) & \Rightarrow &
            z\in X\setminus A\vee w\in Y\setminus B \\
        & \Rightarrow & |f^{-1}(z)|<\alpha\vee|g^{-1}(w)|<\beta \\
        & \Rightarrow & |h^{-1}(z,w)|\leq\min(|f^{-1}(z)|,| g^{-1}(w)|)<\max(\alpha,\beta)
    \end{eqnarray*}}
therefore $(f,g):[0,1]\to X\times Y$ is a $\max(\alpha,\beta)\frac{{\mathcal I}
\times{\mathcal J}}{}$arc, and
\[\Phi([f],[g])=[(f,g)]\in{\mathfrak P}^{\max(\alpha,\beta)}_{{\mathcal I}
\times{\mathcal J}}(X\times Y,(x_0,y_0))\:.\]
(4) (a) is a special
    case of item (3), since $\pi_1(Y,y_0)=
    {\mathfrak P}_{{\mathcal P}(Y)}^\alpha(Y,y_0)$. \\
    For rest note that
    for all $(x,y)\in X\times Y$, continuous maps $f:[0,1]\to X$,
    and $g:[0,1]\to Y$ we have
    $(f,g)^{-1}(x,y)=f^{-1}(x)\cap g^{-1}(y)$, thus
    $|h^{-1}(x,y)|\leq\min(|f^{-1}(x)|,|g^{-1}(y)|)$.
    \begin{itemize}
    \item[(b)] If $f:[0,1]\to X$ is an $\alpha\frac{\mathcal I}{}$arc
        and $g:[0,1]\to Y$ is a $\beta-$arc, then
        for all $(x,y)\in X\times Y$ we have
        $|(f,g)^{-1}(x,y)|\leq\min(|f^{-1}(x)|,|g^{-1}(y)|)\leq|g^{-1}(y)|<\beta$.
        Therefore $(f,g):[0,1]\to X\times Y$ is a $\beta-$arc.
    \item[(c)] If $f:[0,1]\to X$ is an $\alpha-$arc and $g:[0,1]\to Y$ is
        a $\beta-$arc, then
        for all $(x,y)\in X\times Y$ we have
        $|(f,g)^{-1}(x,y)|\leq\min(|f^{-1}(x)|,|g^{-1}(y)|)<\min(\alpha,\beta)$.
        Therefore $(f,g):[0,1]\to X\times Y$ is a $\min(\alpha,\beta)-$arc.
    \end{itemize}
(5) Since ${\mathfrak P}_{\mathcal I}^\alpha(X,t)\subseteq
    {\mathfrak P}_{\mathcal K}^\alpha(X,t)\subseteq
    {\mathfrak P}_{\mathcal K}^{\max(\alpha,\beta)}(X,t)
    \subseteq{\mathfrak P}_{\mathcal K}^{\max(\alpha,\beta)}(X\cup Y,t)$
    and similarly
    ${\mathfrak P}_{\mathcal J}^\beta(Y,t)\subseteq
    {\mathfrak P}_{\mathcal K}^{\max(\alpha,\beta)}(X\cup Y,t)$.
\end{proof}
\begin{example}[$\mathfrak P_{\mathcal I}^\alpha(X)$
for some well-known spaces $X$] We may find the following easy examples:
\begin{itemize}
\item[1.] It's evident that for any contractible space $X$, nonzero cardinal
    number $\alpha$ and ideal $\mathcal I$ on $X$, we have $\mathfrak
    P_{\mathcal I}^\alpha(X)=\{e\}$.
\item[2.] Let $X=\{e^{2\pi i\theta}:\theta\in[0,1]\}(={\mathbb S}^1)$. Then
    $\mathfrak P^\alpha_\mathcal I(X)=\pi_1(X)$, for all $\alpha\geq2$
    and ideal $\mathcal I$ on $X$ (since for $f:[0,1]\to{\mathbb S}^1$
    with $f(t)=e^{2\pi i t}$ we have $[f]\in{\mathfrak
    P}^\alpha({\mathbb S}^1) \subseteq{\mathfrak
    P}^\alpha_{\mathcal I}({\mathbb S}^1)\subseteq\pi_1({\mathbb S}^1)$ and
    $[f]$ is a generator of $\pi_1({\mathbb S}^1)$, thus ${\mathfrak
    P}^\alpha_{\mathcal I}({\mathbb S}^1)=\pi_1({\mathbb S}^1)\cong{\mathbb Z}$).
\item[3.] With a similar method described in item (2), for all
    $\alpha\geq2$ and ideal $\mathcal I$ on ${\mathbb
    R}^2\setminus\{0\}$ (punctured space), we have ${\mathfrak
    P}^\alpha_{\mathcal I}({\mathbb R}^2\setminus\{0\})=\pi_1({\mathbb
    R}^2\setminus\{0\})\cong{\mathbb Z}$.
\item[4.] Using (2) and a similar method described in
    Theorem~\ref{Narges3} for all $\alpha\geq2$ and ideal
    $\mathcal I$ on ${\mathbb T}={\mathbb S}^1\times{\mathbb S}^1$ (Torus) we
    have ${\mathfrak P}^\alpha_{\mathcal I}({\mathbb T})=\pi_1({\mathbb
    T})$.
\end{itemize}
\end{example}
\section{Some preliminaries on Hawaiian earring}
\noindent In this section we bring some useful properties of
Infinite earring (Hawaiian earring) (see \cite[page 500, Exercise
5]{M07} too)
\begin{lemma}\label{lem}
If loop $f:[0,1]\to\mathbb S^1$ is not null-homotopic
and $f(0)=f(1)=1$,
then there exist $a,b\in[0,1]$
such that $f(a,b)=\mathbb S^1\setminus\{1\}$ and $f(a)=f(b)=1$.
\end{lemma}
\begin{proof}
In the following proof for $g:[u,v]\to X$ with $g_{[u,v]}:[0,1]\to X$ we mean
$g_{[u,v]}(t)=g(t(v-u)+u)$.
Since $f:[0,1]\to\mathbb S^1$ is uniformly continuous,there exists
$\varepsilon>0$ such that for all $s,t\in[0,1]$ with $|s-t|<\varepsilon$
we have $|f(s)-f(t)|<\frac12$.
\\
Let $T=\{t\in[0,1]:f(0)=f(t)=1$ and $f_{[0,t]}:[0,1]\to\mathbb
S^1$ is not null-homotopic$\}$. We have $T\neq\varnothing$, since
$1\in T$. Suppose $\tau=\inf(T)$. Since $f$ is continuous and
$T\subseteq f^{-1}(1)$, thus $f(\tau)=1$. We claim that $\tau\in
T$.There exists $t\in T$ such that $0\leq t-\tau<\varepsilon$, if
$\tau=t\in T$ we are done, otherwise since
$f_{[\tau,t]}([0,1])=f[\tau,t]\subseteq \{x\in\mathbb
S^1:|x-1|=|x-f(\tau)|<\frac12\}\subseteq\mathbb
S^1\setminus\{-1\}$, thus $f_{[\tau,t]}$ is null-homotopic. On
the other hand
$[f_{[0,t]}]=[f_{[0,\tau]}]*[f_{[\tau,t]}]=[f_{[0,\tau]}]$ and
$f_{[0,\tau]}$ is not null-homotopic, which indicates $\tau\in T$.
\\
Let $S=\{s\in[0,\tau]:f(s)=f(\tau)=1$ and
$f_{[s,\tau]}:[0,1]\to\mathbb
S^1$ is not null-homotopic$\}$, so $0\in S$ and $S\neq\varnothing$.
let $\sigma=\sup(S)$. Similar to first part of proof, $\sigma\in
S$. It is clear that $\sigma<\tau$. Moreover
$[f_{[0,\tau]}]=[f_{[0,\sigma]}]*[f_{[\sigma,\tau]}]$
and using the way of choose of
$\tau$,
$f_{[0,\sigma]}:[0,1]\to\mathbb
S^1$ is null-homotopic,  thus
$[f_{[0,\tau]}]=[f_{[\sigma,\tau]}]$  and $f_{[\sigma,\tau]}:[0,1]\to\mathbb
S^1$is not null-homotopic.
\\
Since $f_{[\sigma,\tau]}:[0,1]\to\mathbb S^1$ is
not null-homotopic, $f[\sigma,\tau]=f_{[\sigma,\tau]}([0,1])=\mathbb S^1$.
\\
On the other hand if there exists $\zeta\in(\sigma,\tau)$ such that
$f(\zeta)=1$. Respectively using the way of choose of
$\tau$ and $\sigma$, two maps $f_{[0,\zeta]}:[0,1]\to\mathbb
S^1$
and $f_{[\zeta,\tau]}:[0,1]\to\mathbb
S^1$ are null-homotopic. Using
$[f_{[0,\tau]}]=[f_{[0,\zeta]}]*[f_{[\zeta,\tau]}]$,
$f_{[0,\sigma]}[0,1]\to\mathbb
S^1$ is null-homotopic, which
is a contradiction.
Therefore for all
$\zeta\in(\sigma,\tau)$ we have $f(\zeta)\neq1$, which shows
$f(\sigma,\tau)=\mathbb S^1\setminus\{1\}$.
\end{proof}
\begin{lemma}\label{bazlem}
If $X=(\mathbb S^1-1)\cup(\mathbb S^1+1)$ ($X$ and Figure {\large
{\bf 8}} are homeomorph), $\rho:[0,1]\to X$ with $\rho(t)=e^{4\pi
it}-1$ for $t\in[0,\frac12]$ and $\rho(t)=-e^{4\pi it}+1$ for
$t\in[\frac12,1]$, and loop $f:[0,1]\to X$ with $f(0)=f(1)=0$ is homotopic to
$\rho:[0,1]\to X$, then there exist $a,b,c,d\in[0,1]$ such that
$a<b\leq c<d$, $f(a)=f(b)=f(c)=f(d)=0$, $f(a,b)=(\mathbb
S^1-1)\setminus\{0\}$ and $f(c,d)=(\mathbb S^1+1)\setminus\{0\}$.
\end{lemma}
\begin{proof}
Let:
\[f^{\mathbb S^1-1}(t)=\left\{\begin{array}{lc}
    f(t) & f(t)\in{\mathbb S^1-1} \\
    0 & {\rm otherwise}
    \end{array}\right.\SP,\SP
    \rho^{\mathbb S^1-1}(t)=\left\{\begin{array}{lc}
    \rho(t) & \rho(t)\in{\mathbb S^1-1} \\
    0 & {\rm otherwise}
    \end{array}\right.\]
\[f^{\mathbb S^1+1}(t)=\left\{\begin{array}{lc}
    f(t) & f(t)\in{\mathbb S^1+1} \\
    0 & {\rm otherwise}
    \end{array}\right.\SP,\SP
    \rho^{\mathbb S^1+1}(t)=\left\{\begin{array}{lc}
    \rho(t) & \rho(t)\in{\mathbb S^1+1} \\
    0 & {\rm otherwise}
    \end{array}\right.\]
Two maps $f^{\mathbb S^1-1},\rho^{\mathbb S^1-1}:[0,1]\to\mathbb S^1-1$
are homotopic, since $f,\rho:[0,1]\to X$ are homotopic.
Since $\rho^{\mathbb S^1-1}:[0,1]\to\mathbb S^1-1$ is not null-homotopic,
by Lemma~\ref{lem} there exists
$a,b\in[0,1]$ with $f^{\mathbb S^1-1}(a,b)=(\mathbb S^1-1)\setminus\{0\}$ and
$f^{\mathbb S^1-1}(a)=f^{\mathbb S^1-1}(b)=0$.
For all $t\in(a,b)$ we have $f^{\mathbb S^1-1}(t)\neq0$,
therefore $f(t)=f^{\mathbb S^1-1}(t)$. Thus $f(a,b)=f^{\mathbb S^1-1}(a,b)=
{\mathbb S^1-1}\setminus\{0\}$.
Moreover $f^{\mathbb S^1-1}(a)=f^{\mathbb S^1-1}(b)=0$,
thus $f(a),f(b)\in {\mathbb S^1+1}$. Using the continuity of $f$ we have
$f(a),f(b)\in\overline{f(a,b)}={\mathbb S^1-1}$, therefore
$f(a),f(b)\in{\mathbb S^1-1}\cap{\mathbb S^1+1}=\{0\}$ and
$f(a)=f(b)=0$. Let:
\[\Gamma_1:=\{(a,b)\in[0,1]\times[0,1]:f(a,b)=(\mathbb S^1-1)\setminus\{0\},f(a)=f(b)=0\}\:,\]
\[\Gamma_2:=\{(a,b)\in[0,1]\times[0,1]:f(a,b)=(\mathbb S^1+1)\setminus\{0\},f(a)=f(b)=0\}\:.\]
By the above discussion, $\Gamma_1\neq\varnothing$. It is evident
that for all distinct $(a,b),(a',b')\in\Gamma_1$ we have
$(a,b)\cap(a',b')=\varnothing$. Since $f:[0,1]\to X$ is uniformly
continuous there exists $\delta>0$ such that:
\[\forall u,v\in[0,1]\SP(|u-v|<\delta\Rightarrow|f(u)-f(v)|<1)\]
which leads to:
\[\forall u,v\in[0,1]\SP(|u-v|<\delta\Rightarrow f(u,v)\neq {\mathbb S^1-1})\]
so for all $(a,b)\in\Gamma_1$ we have $b-a\geq\delta$.
\\
Thus $\Gamma_1$ is finite,
since $\Gamma_1$ is a nonempty collection of disjoint subintervals
of $[0,1]$ with $b-a\geq\delta$ for all $(a,b)\in\Gamma_1$.
\\
In a similar way $\Gamma_2$ is a nonempty finite collection of disjoint subintervals
of $[0,1]$.
\\
It is evident that for all $(a,b)\in\Gamma_1$ and $(c,d)\in\Gamma_2$
we have $(a,b)\cap(c,d)\neq\varnothing$ (since $f(a,b)\cap f(c,d)=
((\mathbb S^1-1)\cap(\mathbb S^1+1))\setminus\{0\}=\varnothing$),
therefore $a<b\leq c<d$ or $c<d\leq a<b$.
\\
If there exist $(a,b)\in\Gamma_1$ and $(c,d)\in\Gamma_2$ with
 $a<b\leq c<d$, we are done, otherwise suppose
for all $(a,b)\in\Gamma_1$ and $(c,d)\in\Gamma_2$ we have
$c<d\leq a<b$. Let
\[\Gamma_1=\{(a_1,b_1),\ldots,(a_n,b_n)\}\:,\:
    \Gamma_2=\{(c_1,d_1),\ldots,(c_m,d_m)\}\:.\]
and suppose
\[c_1<d_1\leq c_2<d_2\leq\cdots\leq
    c_m<d_m\leq a_1<b_1\leq a_2<b_2\leq\cdots\leq a_n<b_n\:\]
Using the same notations as in Lemma~\ref{lem}, if
$d_1<c_2$, then $f_{[d_1,c_2]}:[0,1]\to X$ is null-homotopic
(use Lemma~\ref{lem} and consider $\Gamma_1,\Gamma_2$).
if $p\in[0,1]$ suppose $f_{[p,p]}:[0,1]\to X$ is constant 0 function.
So
\[f_{[0,c_1]},f_{[d_1,c_2]},f_{[d_2,c_3]},\ldots,f_{[d_{m-1},c_m]},
f_{[d_{m},a_1]},f_{[b_1,a_2]},f_{[b_{n-1},a_n]},f_{[a_n,1]}:[0,1]\to X\]
are null-homotopic. Thus
\[[f]=[f_{[c_1,d_1]}]*\cdots*[f_{[c_m,d_m]}]*[f_{[a_1,b_1]}]*\cdots*[f_{[a_n,b_n]}]\:\]
For all $i,j$ we have $f_{[c_i,d_i]}\subseteq\mathbb S^1+1$
and $f_{[a_j,b_j]}\subseteq\mathbb S^1-1$.
thus there exist $q_1,\ldots,q_m,p_1,\ldots,p_n\geq0$ with
\begin{center}
$[f_{[c_i,d_i]}]=[\rho_{[\frac12,1]}]^{q_i}$  ($1\leq i\leq m$) and
$[f_{[a_j,b_j]}]=[\rho_{[0,\frac12]}]^{p_j}$ ($1\leq j\leq n$)
\end{center}
(we recall that
$\pi_1(X)=\pi_1(S^1-1)*\pi_1(S^1+1)=
<[\rho_{[0,\frac12]}]>*<[\rho_{[\frac12,1]}]>$, by van Kampen Theorem).
Thus
\[[\rho_{[0,\frac12]}]*[\rho_{[\frac12,1]}]=[\rho]=[f]
    =[\rho_{[\frac12,1]}]^{q_1+\cdots+q_m}*[\rho_{[0,\frac12]}]^{p_1+\cdots+p_n}\]
which is a contradiction since $\pi_1(X)$ is nonabelian free group over
two generators $[\rho_{[0,\frac12]}]$ and $[\rho_{[\frac12,1]}]$.
\end{proof}
\begin{lemma}\label{bazbazlem}
If loop $f:[0,1]\to {\mathcal Z}$ with $f(0)=f(1)=0$ is homotopic to
$f_{\mathcal Z}:[0,1]\to {\mathcal Z}$, then there exist $s_1,t_1,s_2,t_2,\ldots,s_p,t_p\in[0,1]$
such that
$s_1<t_1\leq s_2<t_2\leq\ldots\leq s_p<t_p$, $f(s_j)=f(t_j)=0$, $f(s_j,t_j)=
\{\frac1je^{2\pi i(x-j-\frac14)}+\frac{i}{j}:x\in[0,1]\}\setminus\{0\}$
for all $j\in\{1,\ldots,p\}$.
\end{lemma}
\begin{proof}
Use the same method described in Lemma~\ref{bazlem} and
note to the fact that $\pi_1(\mathcal Z)$ is nonabelian free group
over $p $ generators $[h_{[\frac{k-1}p,\frac{k}p]}]$ for
$k=1,\ldots,p$ where $h:=f_{\mathcal Z}$ and using the notations of
Lemma~\ref{lem}.
\end{proof}
\begin{note}\label{Narges4}
Consider loops $f,g:[0,1]\to\mathcal X$ such that
$f(0)=f(1)=g(0)=g(1)=0$. For nonempty subset
$\Gamma$ of $\mathbb N$ and $h:[0,1]\to \mathcal X$ let:
    \[h^\Gamma(x)=\left\{\begin{array}{lc} h(x) &
                h(x)\in\bigcup\{C_n:n\in\Gamma\}\:, \\
             0 & h(x)\in ({\mathcal X}\setminus\bigcup\{C_n:n\in\Gamma\})\cup\{0\}\:.
     \end{array}\right.\]
(As a matter of fact we denote $h^{\bigcup\{C_n:n\in\Gamma\}}$
(see Convention~\ref{good10}) briefly by $h^\Gamma$)
\begin{itemize}
\item[1.] If $f,g:[0,1]\to\mathcal X$
    are homotopic, then $f^\Gamma,g^\Gamma:[0,1]\to \bigcup\{C_n:n\in\Gamma\}$ are homotopic
        (equivalently
    $f^\Gamma,g^\Gamma:[0,1]\to\mathcal X$ are homotopic).
\item[2.] For loop $h:[0,1]\to\mathcal X$ with $h(0)=h(1)=0$ define:
    \[A(h):=\{n\in\mathbb N: h^{\{n\}}:[0,1]\to C_n\:{\rm is \: not \: null-homotopic}\}\:.\]
    Then $A(h)$ is a subset of:
    \[\{n\in{\mathbb N}:\exists a,b\in[0,1]\:
        (h(a,b)=C_n\setminus\{0\}\wedge h(a)=h(b)=0)\}\:.\]
    Moreover if $f,g:[0,1]\to\mathcal X$ are
    homotopic, then $A(f)=A(g)$.
\item[3.] For loop $h:[0,1]\to\mathcal X$ with $h(0)=h(1)=0$,
    we have $|h^{-1}(0)|\geq|A(h)|$.
\item[4.] If $[f]\in\mathfrak P^\omega(\mathcal X)$,
    then $|A(f)|<\omega$ and $A(f)$ is finite.
\end{itemize}
\end{note}
\begin{proof} $\:$
\begin{itemize}
\item[1.]
Note to the fact that $A=\bigcup\{C_n:n\in\Gamma\}$ and
    $B=(\mathcal X\setminus A)\cup\{0\}$ are closed (linear connected)
    subsets of $\mathcal X$. Moreover $A\cap B=\{0\}$. Now
use the same argument as in Convention~\ref{good10}.
\item[2.]
If $n\in A(h)$, then $h^{\{n\}}:[0,1]\to C_n$ is not null-homotopic.
By Lemma~\ref{lem} there exist $a,b\in[0,1]$  with $h(a)=h(b)=0$
and $h(a,b)=C_n\setminus\{0\}$. Use item (1) to complete the proof.
\item[3.] By (2) for all $n\in A(h)$ there exists
    $a_n,b_n\in[0,1]$ with $h(a_n,b_n)=C_n\setminus\{0\}$ and
    $h(a_n)=h(b_n)=0$.
    We claim that $\mathop{A(h)\to h^{-1}(0)}\limits_{n\mapsto a_n\SP}$
    is one to one. Suppose $n\neq m$ and $n,m\in A(h)$. By
    \[h(a_n,b_n)\cap h(a_m,b_m)=(C_n\setminus\{0\})\cap(C_m\setminus\{0\})
    =\varnothing\]
    we have $(a_n,b_n)\cap(a_m,b_m)=\varnothing$, thus
    $a_n\neq a_m$.
\item[4.] If
    $[f]\in\mathfrak P^\omega(\mathcal X)$, then by
    Note~\ref{Narges2} there exists $\omega-$loop
    $k:[0,1]\to\mathcal X$ with $k(0)=k(1)=0$ homotopic to
    $f:[0,1]\to\mathcal X$. By (3)
    we have $|A(k)|\leq|k^{-1}(0)|<\omega$. By item (2)
    we have $A(f)=A(k)$ which leads to
    $|A(f)|=|A(k)|\leq|k^{-1}(0)|<\omega$.
\end{itemize}
\end{proof}
\begin{note}\label{Narges5}
For $(m,n)\in\mathbb N\times\mathbb N$ and loop $h:[0,1]\to\mathcal Y$
with base point $0$, define:
\[h^{(m,n)}(t)=\left\{\begin{array}{lc} h(t) & h(t)\in\dfrac1{2^{m+1}}C_n+\dfrac1m \: , \\
& \\
    \dfrac1m & h(t)\notin\dfrac1{2^{m+1}}C_n+\dfrac1m \: . \end{array}\right.\]
(As a matter of fact we denote $h^{\frac1{2^{m+1}}C_n+\frac1m}$
(see Convention~\ref{good10}) briefly by $h^{(m,n)}$)
\\
Moreover for loop $h:[0,1]\to\mathcal Y$ we define:
{\small \[B(h):=\left\{(m,n)\in{\mathbb N}\times{\mathbb
N}:h^{(m,n)}:[0,1]\to\dfrac1{2^{m+1}}C_n+\dfrac1m\:{\rm is \: not \:
null-homotopic}\right\}\:,\]}
then $B(h)$ is a subset of:
{\small\[\left\{(m,n)\in{\mathbb N}\times{\mathbb
N}:\exists a,b\in[0,1]\:
(h(a,b)=(\dfrac1{2^{m+1}}C_n+\dfrac1m)\setminus\{\dfrac1m\}\wedge
h(a)=h(b)=\dfrac1m)\right\}\]}
and for loops $f,g:[0,1]\to\mathcal Y$ with
$f(0)=f(1)=g(0)=g(1)=0$, we have: \\
\\
1. If $f,g:[0,1]\to\mathcal Y$ are homotopic, then
    $B(f)=B(g)$.
\\
2. For $m\in\mathbb N$, we have:
    \begin{itemize}
    \item[a.] $|f^{-1}(\frac1m)|\geq|\{n\in\mathbb N:(m,n)\in B(f)\}|$.
    \item[b.] If $[f]\in\mathfrak P^\omega(\mathcal Y)$, then
        $|\{n\in\mathbb N:(m,n)\in B(f)\}|<\omega$.
    \item[c.] If $\mathcal I$ is an ideal on $\mathcal Y$ and
        $[f]\in\mathfrak P^\omega_{\mathcal I}(\mathcal Y)$,
        then there exists $F\in\mathcal I$ such that for all
        $k\in\mathbb N$ with $\frac1k\in\mathcal Y\setminus F$, we have
        $|\{n\in\mathbb N:(k,n)\in B(f)\}|<\omega$.
    \end{itemize}
\end{note}
\begin{proof} If $(m,n)\in B(h)$, then $h^{(m,n)}:[0,1]\to\frac1{2^{m+1}}C_n+\frac1m$ is
not null-homotopic, thus by Lemma~\ref{lem} there exist $a,b\in[0,1]$ with
$h(a,b)=(\frac1{2^{m+1}}C_n+\frac1m)\setminus\{\frac1m\}$ and
$h(a)=h(b)=\frac1m$.
\\
1) Suppose $f,g:[0,1]\to\mathcal Y$ are homotopic. For
$m,n\in\mathbb N$, two sets $\frac1{2^{m+1}}C_n+\frac1m$ and
$({\mathcal
Y}\setminus(\frac1{2^{m+1}}C_n+\frac1m))\cup\{\frac1m\}$  are
closed (linear) subsets of $\mathcal Y$ with
$(\frac1{2^{m+1}}C_n+\frac1m)\cap(({\mathcal
Y}\setminus(\frac1{2^{m+1}}C_n+\frac1m))\cup\{\frac1m\})=\{\frac1m\}$.
Thus using the same argument as in Convention~\ref{good10} (note
to the fact that base point in the proof of
Convention~\ref{good10} is not important) two maps
$f^{(m,n)},g^{(m,n)}:[0,1]\to\frac1{2^{m+1}}C_n+\frac1m$ are
homotopic, therefore $(m,n)\in B(f)$ if and only if $(m,n)\in
B(g)$. So $B(f)=B(g)$.
\\
2-a) For all $(m,n)\in B(f)$ there exists
$a_{(m,n)},b_{(m,n)}\in[0,1]$ with
$f(a_{(m,n)},b_{(m,n)})=(\frac1{2^{m+1}}C_n+\frac1m)\setminus\{\frac1m\}$ and
$f(a_{(m,n)})=f(b_{(m,n)})=\frac1m$.
We claim that
\[\mathop{\{n\in{\mathbb N}:(m,n)\in B(f)\}\to f^{-1}(\dfrac1m)}\limits_{
\SP\SP\SP\SP\SP\SP\SP\SP\SP n\mapsto a_{(m,n)}}\]
is one to one. Suppose $n\neq k$ and $(m,k),(m,n)\in B(f)$. By
    \[f(a_{(m,n)},b_{(m,n)})\cap
    f(a_{(m,k)},b_{(m,k)})\]
    \[=((\dfrac1{2^{m+1}}C_n+\dfrac1m)\setminus\{\dfrac1m\})
    \cap((\dfrac1{2^{m+1}}C_k+\dfrac1m)\setminus\{\dfrac1m\})
    =\varnothing\]
 we have $(a_{(m,n)},b_{(m,n)})\cap(a_{(m,k)},b_{(m,k)})=\varnothing$, thus
    $a_{(m,n)}\neq a_{(m,k)}$. Therefore $|f^{-1}(\frac1m)|\geq|\{n\in{\mathbb N}:(m,n)\in B(f)\}|$
\\
2-b) This item is a special case of (c) for ${\mathcal I}=\{\varnothing\}$.
\\
2-c) If $[f]\in\mathfrak P^\omega_{\mathcal I}(\mathcal Y)$,
    then by
    Note~\ref{Narges2} there exists $\omega\frac{\mathcal I}{}$loop
    $h:[0,1]\to\mathcal Y$ homotopic to
    $f:[0,1]\to\mathcal Y$
    also we may suppose $h(0)=h(1)=0$. There exists
    $F\in\mathcal I$ such that for all $z\in{\mathcal Y}\setminus
    F$ we have $|h^{-1}(z)|<\omega+1$. In particular for all
    $k\in\mathbb N$ with $\frac1k\in{\mathcal Y}\setminus F$ we
    have $|h^{-1}(\frac1k)|<\omega$, which leads to
    $|\{n\in\mathbb N:(k,n)\in B(h)\}|\leq|h^{-1}(\frac1k)|<\omega$ by (a).
    Using (1) we have $B(f)=B(h)$, thus
    $|\{n\in\mathbb N:(k,n)\in B(h)\}|=|\{n\in\mathbb N:(k,n)\in B(f)\}|<\omega$.
\end{proof}
\section{$\mathfrak
P^c(\mathcal X)$ is a proper subset of $\pi_1(\mathcal X)$}
\noindent Here we want to prove $\pi_1(\mathcal X)\setminus \mathfrak
P^c(\mathcal X)\neq\varnothing$ step by step.
\\
Consider the following conventions in this section:
\\
Usually in order
to construct Cantor set, one may remove the following intervals
step by step from $[0,1]$:
\begin{center}
\scalebox{0.9}{
$\begin{array}{l}
    \: \\
    (c_1^1,d_1^1)=(\frac13,\frac23) \: , \\
    \: \\
    (c_2^1,d_2^1)=(\frac19,\frac29)\:,\:(c_2^2,d_2^2)=(\frac79,\frac89) \:, \\
    \SP \vdots \\
    (c_n^1,d_n^1)=(\frac1{3^n},\frac2{3^n})\:,\:(c_n^2,d_n^2)=(\frac23+\frac1{3^n},\frac23+\frac2{3^n})
        \:,\:\cdots\:,\:(c_n^{2^{n-1}},d_n^{2^{n-1}})=(1-\frac2{3^n},1-\frac1{3^n})\:, \\
    \SP \vdots \\ \: \end{array}$}
\end{center}
So $M=[0,1]\setminus\bigcup\{(c_n^i,d_n^i):n\in\mathbb
N,i\in\{1,...,2^{n-1}\}\}$ is Cantor set. Now suppose:
\begin{center}
\scalebox{0.85}{
$\begin{array}{l}
    \: \\
    (a_1,b_1)=(c_1^1,d_1^1) \:, \\ \: \\
    (a_2,b_2)=(c_2^1,d_2^1)\:,\: (a_3,b_3)=(c_2^2,d_2^2)\:, \\
    \SP \vdots \\
    (a_{2^{n-1}},b_{2^{n-1}})=(c_n^1,d_n^1)\:,\:(a_{2^{n-1}+1},b_{2^{n-1}+1})=(c_n^2,d_n^2)
        \:,\: \cdots \:,\: (a_{2^n-1},b_{2^n-1})=(c_n^{2^{n-1}},d_n^{2^{n-1}})\:, \\
    \SP\vdots \\ \: \end{array}$}
\end{center}
\textbf{Define $g:[0,1]\to{\mathcal X}$ with:
\[g(x)=\left\{\begin{array}{lc}
    \dfrac1ne^{2\pi i\frac{x-a_n}{b_n-a_n}}+\dfrac{i}{n}& x\in(a_n,b_n),n\in\mathbb N \\
    &\\
    0 & {\rm otherwise} \end{array}\right.\]
Suppose the loops $g,f:[0,1]\to\mathcal X$ are homotopic with
$f(0)=f(1)=0$. Consider the above mentioned $f$ and $g$ in this section.}
\\
It is well-known that (see \cite{Rudin}):
\[M=\left\{{\displaystyle\sum_{n=1}^\infty\frac{x_n}{3^n}}:
    \forall n\in\mathbb N\SP x_n\in\{0,2\}\right\}\:.\]
For $x={\displaystyle\sum_{n=1}^\infty\frac{x_n}{3^n}}\in M$ with $x_n\in\{0,2\}$ ($n\in\mathbb N$).
For $m\in\mathbb N$ choose $n^x_m\in\mathbb N$ such that:
\[a_{n^x_m}=\left\{\begin{array}{lc}
    \min\{c_m^i:1\leq i\leq2^{m-1},x\leq c_m^i\} & x_m=0 \\
    \max\{c_m^i:1\leq i\leq2^{m-1},c_m^i\leq x\} & x_m=2
    \end{array}\right.\]
also let
\[E^x:=\{n\in{\mathbb N}:x_n=0\}\:,\:F^x:=\{n\in{\mathbb N}:x_n=2\}\:.\]
Finally consider:
\[K:=\{x\in M: E^x\:{\rm and}\: F^x\:{\rm are \: infinite}\}\:.\]
We have the following sequel of lemmas and notes.
\begin{lemma}\label{note-cantor}
For $x={\displaystyle\sum_{n=1}^\infty\frac{x_n}{3^n}}\in M$ with $x_n\in\{0,2\}$ ,
we have:
\[a_{n_k^x}=\left\{\begin{array}{lc}
    {\displaystyle\sum_{n=1}^k\frac{x_n}{3^n}}+\dfrac1{3^k} & x_k=0 \: , \\
    {\displaystyle\sum_{n=1}^k\frac{x_n}{3^n}}-\dfrac1{3^k} & x_k=2 \: , 
    \end{array}\right.
    \SP{\rm and}\SP
b_{n_k^x}=\left\{\begin{array}{lc}
    {\displaystyle\sum_{n=1}^k\frac{x_n}{3^n}}+\dfrac2{3^k} & x_k=0 \: , \\
    {\displaystyle\sum_{n=1}^k\frac{x_n}{3^n}} & x_k=2 \: .
    \end{array}\right.\tag{*}\]
And:
\[|a_{n_k^x}-x|\leq\dfrac2{3^k}\SP{\rm and}\SP|b_{n_k^x}-x|\leq\dfrac2{3^k}\:
({\rm for\: all\:}k\in\mathbb {N})\:.\tag{**}\]
\end{lemma}
\begin{proof}
For each $k\in\mathbb N$ suppose
\[A_k=\left\{{\displaystyle\sum_{n=1}^k\frac{y_n}{3^n}}:y_1,\ldots,y_k\in\{0,2\}\right\}\:,\]
then we may suppose $A_k=\{w_k^1,\ldots,w_k^{2^k}\}$ with $w_k^1<w_k^2<\cdots<w_k^{2^k}$.
It is easy to see that:
\[\begin{array}{lcl}
c^1_k=w_k^1+{\displaystyle\sum_{n\geq k+1}\frac{2}{3^n}}=w_k^1+\dfrac1{3^k} & , & d^1_k=w_k^2 \\
c^2_k=w_k^3+{\displaystyle\sum_{n\geq k+1}\frac{2}{3^n}}=w_k^3+\dfrac1{3^k} & , & d^2_k=w_k^4 \\
\vdots && \\
c^i_k=w_k^{2i-1}+{\displaystyle\sum_{n\geq k+1}\frac{2}{3^n}}=w_k^{2i-1}+\dfrac1{3^k} & ,
    & d^i_k=w_k^{2i} \\
\vdots && \\
c^{2^{k-1}}_k=w_k^{2^k-1}+{\displaystyle\sum_{n\geq k+1}\frac{2}{3^n}}=w_k^{2^k-1}+\dfrac1{3^k}
    & , & d^{2^{k-1}}_k=w_k^{2^k}
\end{array}\]
so $(c^i_k,d^i_k)=(w_k^{2i-1}+\frac1{3^k},w_k^{2i})$.
\\
Now for $x={\displaystyle\sum_{n\in\mathbb N}\frac{x_n}{3^n}}\in M$ with
$x_1,x_2,\ldots\in\{0,2\}$ we have:
\begin{itemize}
\item For $p\in\mathbb N$ we have $x_p=0$ and $x_{p+1}=x_{p+2}=\cdots=2$
    if and only if there exists $i\in\{1,\ldots,2^{p-1}\}$ with $x=c^i_p$.
\item For $p\in\mathbb N$ we have $x_p=2$ and $x_{p+1}=x_{p+2}=\cdots=0$
    if and only if there exists $i\in\{1,\ldots,2^{p-1}\}$ with $x=d^i_p$.
\item $x\in K$ if and only if for all $p\in\mathbb N$ we have
    $p\notin\{c^i_p:1\leq i\leq2^{p-1}\}\cup\{d^i_p:1\leq i\leq2^{p-1}\}$
    (and $x\in M$).
\end{itemize}
In particular if $x_k=0$, then
$a_{n_k^x}={\displaystyle\sum_{n=1}^k\frac{x_n}{3^n}}+\dfrac1{3^k}$,
in other words if $w_k^i={\displaystyle\sum_{n=1}^k\frac{x_n}{3^n}}$,
then $i=2j-1$ is odd and
$a_{n_k^x}=w_k^{2j-1}+\dfrac1{3^k}=
{\displaystyle\sum_{n=1}^k\frac{x_n}{3^n}}+\dfrac1{3^k}=c^i_k$.
Also if $x_k=2$, then ${\displaystyle\sum_{n=1}^k\frac{x_n}{3^n}}\in A_k$
and there exists even $i=2j$ such that ${\displaystyle\sum_{n=1}^k\frac{x_n}{3^n}}=w_k^{2j}$,
moreover $b_{n_k^x}=w_k^{2j}$. So we have (*),
moreover considering the following inequalities will complete the proof:
\[|a_{n_k^x}-x|  \leq \left|a_{n_k^x}- {\displaystyle\sum_{n=1}^k\frac{x_n}{3^n}}\right|+
    {\displaystyle\sum_{n=k+1}^\infty\frac{x_n}{3^n}}
 \leq  \dfrac1{3^k}+{\displaystyle\sum_{n=k+1}^\infty\frac2{3^n}}=\dfrac2{3^k}\:,\]
 and:
\begin{eqnarray*}
|b_{n_k^x}-x| & = & \left|{\displaystyle\sum_{n=1}^{k-1}\frac{x_n}{3^n}}+\dfrac2{3^k}-x\right| \\
& = & \left|\dfrac{2-x_k}{3^k}-{\displaystyle\sum_{n=k+1}^\infty\frac{x_n}{3^n}}\right| \\
& = & \left\{\begin{array}{lc}
{\displaystyle\sum_{n=k+1}^\infty\frac{x_n}{3^n}}\leq
{\displaystyle\sum_{n=k+1}^\infty\frac{2}{3^n}}=\dfrac1{3^k} & x_k=2 \\
\dfrac{2}{3^k}-{\displaystyle\sum_{n=k+1}^\infty\frac{x_n}{3^n}}\leq\dfrac2{3^k} & x_k=0
\end{array}\right.
\end{eqnarray*}
which shows (**).
\end{proof}
\begin{lemma}\label{sara1-20}
Let $x={\displaystyle\sum_{n\in\mathbb N}\frac{x_n}{3^n}}\in M$ with
$x_n\in\{0,2\}$ ($n\in\mathbb N$), then we have:
\begin{itemize}
\item[1.] ${\displaystyle\lim_{k\to\infty}a_{n_k^x}}=
    {\displaystyle\lim_{k\to\infty}b_{n_k^x}}=x$.
\item[2.] For $i<j$ if $x_i=x_j=0$, then
    $x\leq a_{n^x_j}<b_{n^x_j}<a_{n^x_i}<b_{n^x_i}$.
\item[3.] For $i<j$ if $x_i=x_j=2$, then $a_{n^x_i}<b_{n^x_i}<a_{n^x_j}<b_{n^x_j}\leq x$.
\end{itemize}
\end{lemma}
\begin{proof}
Use (**) in Lemma~\ref{note-cantor} in order to prove (1).
\\
2) Suppose $i<j$ and $x_i=x_j=0$, then by (*) in Lemma~\ref{note-cantor}
we have:
\begin{eqnarray*}
b_{n^x_i}& > & a_{n^x_i}={\displaystyle\sum_{n=1}^i\frac{x_n}{3^n}}+\frac1{3^i}
    ={\displaystyle\sum_{n=1}^{i}\frac{x_n}{3^n}}+
    {\displaystyle\sum_{n=i+1}^\infty\frac{2}{3^n}}
    > {\displaystyle\sum_{n=1}^{i}\frac{x_n}{3^n}}
    +{\displaystyle\sum_{n=i+1}^{j}\frac{2}{3^n}} \\
& > & {\displaystyle\sum_{n=1}^{i}\frac{x_n}{3^n}}+
    {\displaystyle\sum_{n=i+1}^{j-1}\frac{x_n}{3^n}}+\frac2{3^j}
    ={\displaystyle\sum_{n=1}^{j}\frac{x_n}{3^n}}+\frac2{3^j}
    =b_{n^x_j}>a_{n^x_j}={\displaystyle\sum_{n=1}^j\frac{x_n}{3^n}}+\frac1{3^j} \\
& = & {\displaystyle\sum_{n=1}^j\frac{x_n}{3^n}}+
    {\displaystyle\sum_{n=j+1}^\infty\frac{2}{3^n}}\geq
    {\displaystyle\sum_{n=1}^j\frac{x_n}{3^n}}+
    {\displaystyle\sum_{n=j+1}^\infty\frac{x_n}{3^n}}
    =x
\end{eqnarray*}
3) Use a similar method described in the proof of (2), to prove (3).
\end{proof}
\begin{lemma}\label{sara10}
There exists a sequence $((p_n,q_n):n\in{\mathbb N})$ such that
for all $n,m\in\mathbb N$ we have:
\begin{itemize}
\item $0\leq p_n<q_n\leq1$, $f(p_n,q_n)=C_n\setminus\{0\}$ and
    $f(p_n)=f(q_n)=0$;
\item if
    $a_n<b_n<a_m<b_m$, then $p_n<q_n<p_m<q_m$.
\end{itemize}
\end{lemma}
\begin{proof}
For all $n\in{\mathbb N}$, by Note~\ref{Narges4} we have
$f^{\{n\}},g^{\{n\}}:[0,1]\to C_n$ are homotopic loops, therefore
$f^{\{n\}}:[0,1]\to C_n$ is not null-homotopic. By
Lemma~\ref{lem} there exist $a,b\in[0,1]$ with
$f^{\{n\}}(a,b)=C_n\setminus\{0\}$ and
$f^{\{n\}}(a)=f^{\{n\}}(b)=0$, therefore
$f(a,b)=C_n\setminus\{0\}$ and $f(a)=f(b)=0$. On the other hand
$f:[0,1]\to\mathcal X$ is uniformly continuous, thus
\[\Gamma_n:=\{(a,b)\in[0,1]\times[0,1]:
f(a,b)=C_n\setminus\{0\},f(a)=f(b)=0\}\]
is a finite nonempty set. For $k\in\mathbb N$,
by considering $f^{\{1,\ldots,k\}}:[0,1]\to C_1\cup\cdots\cup C_k$,
Note~\ref{Narges4} and Lemma~\ref{bazbazlem} there exist
$(u_1,v_1)\in\Gamma_1,\ldots,(u_k,v_k)\in\Gamma_k$
such that if $a_i<b_i<a_j<b_j$, then $u_i<v_i\leq u_j<v_j$
for all $i,j\in\{1,\ldots,k\}$.
\\
Using the above mentioned note and finiteness of $\Gamma_1$,
there exists $(p_1,q_1)\in\Gamma_1$ such that
$\sup\{k\in{\mathbb N}:$ there exist $u_2,v_2,u_3,v_3,\ldots,u_k,v_k\in[0,1]$
such that for $u_1=p_1$ and $v_1=q_1$ and all $i,j\in\{1,\ldots,k\}$ we have $(u_i,v_i)\in\Gamma_i$ and if
$a_i<b_i<a_j<b_j$, then $u_i<v_i\leq u_j<v_j\}=\infty$.
\\
For $m\in\mathbb N$ if $(p_1,q_1)\in\Gamma_1,...,(p_m,q_m)\in\Gamma_m$ are
such that $\sup\{k\in{\mathbb N}:$ there exist $u_{m+1},v_{m+1},u_{m+2},v_{m+2},
\ldots,u_k,v_k\in[0,1]$
such that for $u_1=p_1,v_1=q_1,u_2=p_2,v_2=q_2,\ldots,u_m=p_m,v_m=q_m$
for all $i,j\in\{1,\ldots,k\}$ we have $(u_i,v_i)\in\Gamma_i$ and if
$a_i<b_i<a_j<b_j$, then $u_i<v_i\leq u_j<v_j\}=\infty$.
Since $\Gamma_{m+1}$ is finite, there exists
$(p_{m+1},q_{m+1})\in\Gamma_{m+1}$ such that
$\sup\{k\in{\mathbb N}:$ there exist $u_{m+2},v_{m+2},u_{m+3},v_{m+3},
\ldots,u_k,v_k\in[0,1]$
such that for $u_1=p_1,v_1=q_1,u_2=p_2,v_2=q_2,\ldots,u_{m+1}=p_{m+1},v_{m+1}=q_{m+1}$
for all $i,j\in\{1,\ldots,k\}$ we have $(u_i,v_i)\in\Gamma_i$ and if
$a_i<b_i<a_j<b_j$, then $u_i<v_i\leq u_j<v_j\}=\infty$.
\\
The sequence $((p_n,q_n):n\in\mathbb N)$ is our desired sequence.
\end{proof}
\begin{lemma}\label{sara20}
Let $x={\displaystyle\sum_{n\in\mathbb N}\frac{x_n}{3^n}}\in K(\subset M)$ with
$x_n\in\{0,2\}$ ($n\in\mathbb N$),
and
\[E^x=\{n\in{\mathbb N}:x_n=0\}=\{u_k:k\in\mathbb N\}\:,\]
\[F^x=\{n\in{\mathbb N}:x_n=2\}=\{v_k:k\in\mathbb N\}\]
such that $u_1<u_2<\cdots$ and $v_1<v_2<\cdots$, and consider the sequence
$((p_n,q_n):n\in\mathbb N)$ as in Lemma~\ref{sara10},
then we have:
\\
1. The sequences $\{a_{n^x_{u_k}}:k\in\mathbb N\}$
        and $\{b_{n^x_{u_k}}:k\in\mathbb N\}$
        are strictly decreasing to $x$.
\\
2. The sequences $\{a_{n^x_{v_k}}:k\in\mathbb N\}$
        and $\{b_{n^x_{v_k}}:k\in\mathbb N\}$
        are strictly increasing to $x$.
\\
3. The sequences $\{p_{n^x_{u_k}}:k\in\mathbb N\}$
        and $\{q_{n^x_{u_k}}:k\in\mathbb N\}$
        are strictly decreasing.
\\
4. The sequences
        $\{p_{n^x_{v_k}}:k\in\mathbb N\}$ and
        $\{q_{n^x_{v_k}}:k\in\mathbb N\}$
        are strictly increasing.
\\
5.
    ${\displaystyle\lim_{k\to\infty}p_{n^x_{v_k}}
        =\lim_{k\to\infty}q_{n^x_{v_k}}}\leq
        {\displaystyle\lim_{k\to\infty}p_{n^x_{u_k}}
        =\lim_{k\to\infty}q_{n^x_{u_k}}}$.
\end{lemma}
\begin{proof}$\:$ \\
Use Lemma~\ref{sara1-20} in order to prove (1) and (2).
\\
3) By Lemma~\ref{sara1-20} (2), for all $k\in\mathbb N$ we have
\[a_{n^x_{u_{k+1}}}<b_{n^x_{u_{k+1}}}
<a_{n^x_{u_k}}<b_{n^x_{u_k}}\:,\]
which leads to
$p_{n^x_{u_{k+1}}}<q_{n^x_{u_{k+1}}}
<p_{n^x_{u_k}}<q_{n^x_{u_k}}$.
\\
4) Lemma~\ref{sara1-20} (3), for all $k\in\mathbb N$ we have
\[a_{n^x_{v_k}}<b_{n^x_{v_k}}
<a_{n^x_{v_{k+1}}}<b_{n^x_{v_{k+1}}}\:,\]
which leads to
$p_{n^x_{v_k}}<q_{n^x_{v_k}}
<p_{n^x_{v_{k+1}}}<q_{n^x_{v_{k+1}}}$.
\\
5) Using (3) and (4), we have
${\displaystyle\lim_{k\to\infty}p_{n^x_{u_k}}
=\lim_{k\to\infty}q_{n^x_{u_k}}}$
and ${\displaystyle\lim_{k\to\infty}p_{n^x_{v_k}}
=\lim_{k\to\infty}q_{n^x_{v_k}}}$. On
the other hand for all $k\in\mathbb N$ we have
$a_{n^x_{v_k}}<b_{n^x_{v_k}}<x<a_{n^x_{u_k}}<b_{n^x_{u_k}}$, thus
\[p_{n^x_{v_k}}<q_{n^x_{v_k}}<p_{n^x_{u_k}}<q_{n^x_{u_k}}\:,\]
which leads to
${\displaystyle\lim_{k\to\infty}p_{n^x_{v_k}}
\leq\lim_{k\to\infty}p_{n^x_{u_k}}}$.
\end{proof}
\begin{lemma}\label{sara30}
For $x={\displaystyle\sum_{n\in\mathbb N}\frac{x_n}{3^n}}\in K$ with $x_n\in\{0,2\}$
and $E^x=\{n\in{\mathbb N}:x_n=0\}=\{u_k:k\in\mathbb N\}$ with $u_1<u_2<\cdots$
under the same notations as in Lemma~\ref{sara10}, by Lemma~\ref{sara20},
$\{p_{n^x_{u_k}}:k\in\mathbb N\}$ is an strictly decreasing sequence (in $[0,1]$).
Let $\eta(x)={\displaystyle\lim_{k\to\infty}p_{n^x_{u_k}}}$, then
$\eta:K\to[0,1]$ is strictly increasing, and
for all $x\in K$ we have $f(\eta(x))=0$.
\end{lemma}
\begin{proof}
Consider $x,y\in K$ with $x<y$. Suppose
$x={\displaystyle\sum_{n\in\mathbb N}\frac{x_n}{3^n}}$ and
$y={\displaystyle\sum_{n\in\mathbb N}\frac{x_n}{3^n}}$ with $x_n,y_n\in\{0,2\}$
for all $n\in\mathbb N$. Let
\[E^x=\{n\in{\mathbb N}:x_n=0\}=\{u_k:k\in\mathbb N\}\:,\]
\[E^y=\{n\in{\mathbb N}:y_n=0\}=\{u'_k:k\in\mathbb N\}\:,\]
with
\[u_1<u_2<\cdots\SP{\rm and}\SP u'_1<u'_2<\cdots\SP.\]
By Lemma~\ref{sara20} (1), $\{a_{n_{u_k}^x}:k\in\mathbb N\}$ is a strictly
decreasing sequence to $x$, and $\{a_{n_{u'_k}^x}:k\in\mathbb N\}$ is a strictly
decreasing sequence to $y$. Since $x<y$, there exists
$m\in\mathbb N$ such that
\[x\leq\cdots<a_{n_{u_{m+2}}^x}<a_{n_{u_{m+1}}^x}<a_{n_{u_m}^x}<y
\leq\cdots<a_{n_{u'_{m+2}}^y}<a_{n_{u'_{m+1}}^y}<a_{n_{u'_m}^y}\:.\]
Thus
\begin{eqnarray*}
x & \leq & \cdots<a_{n_{u_{m+2}}^x}<b_{n_{u_{m+2}}^x}<a_{n_{u_{m+1}}^x}<
        b_{n_{u_{m+1}}^x}<a_{n_{u_m}^x}<b_{n_{u_{m}}^x} \\
& < & y\leq\cdots<a_{n_{u'_{m+2}}^y}<b_{n_{u'_{m+2}}^y}<a_{n_{u'_{m+1}}^y}
    <b_{n_{u'_{m+1}}^y}<a_{n_{u'_m}^y}<b_{n_{u'_{m}}^y}\:.
\end{eqnarray*}
Using Lemma~\ref{sara10} we have:
\[\begin{array}{c}
\cdots<p_{n_{u_{m+2}}^x}<q_{n_{u_{m+2}}^x}<p_{n_{u_{m+1}}^x}<
        q_{n_{u_{m+1}}^x}<p_{n_{u_m}^x}<q_{n_{u_{m}}^x} \\
<\cdots<p_{n_{u'_{m+2}}^y}<q_{n_{u'_{m+2}}^y}<p_{n_{u'_{m+1}}^y}
    <q_{n_{u'_{m+1}}^y}<p_{n_{u'_m}^y}<q_{n_{u'_{m}}^y}\:.
\end{array}\]
Therefore
\[\eta(x)={\displaystyle\lim_{k\to\infty}p_{n^x_{u_k}}}\leq p_{n^x_{u_{m+1}}}
< p_{n^x_{u_m}}\leq {\displaystyle\lim_{k\to\infty}p_{n^y_{u'_k}}}=\eta(y)\:,\]
and $\eta:K\to[0,1]$ is strictly increasing.
Since $f(p_{n^x_{u_k}})=0$ for all $k\in\mathbb N$ and $f$ is continuous,
we have $f(\eta(x))=0$.
\end{proof}
\begin{lemma}\label{sara40}
$|f^{-1}(0)|\geq c$ and $f$ is not a $c-$arc.
\end{lemma}
\begin{proof}
Consider $\eta:K\to[0,1]$ as in Lemma~\ref{sara30}.
By Lemma~\ref{sara30} we have
$|f^{-1}(0)|\geq|\eta(K)|$ and by
Lemma~\ref{sara20} $\eta$ is one to one, therefore
$|\eta(K)|=|K|=c$. Thus $|f^{-1}(0)|\geq c$ and $f$ is not a $c-$arc.
\end{proof}
\begin{theorem}\label{taha20}
We have
\[{\mathfrak P}^c(\mathcal X)\subset\pi_1(\mathcal X)\:, \:
\mathfrak P_{\mathcal P_{fin}(\mathcal X)}^\omega
(\mathcal X)\not\subseteq{\mathfrak P}^c(\mathcal X)\:,\]
\[\mathfrak P^\omega
(\mathcal X)\subset\mathfrak P_{\mathcal P_{fin}(\mathcal X)}^\omega
(\mathcal X)\:{\rm and}\:\mathfrak P_{\mathcal P_{fin}(\mathcal X)}^\omega
(\mathcal X)\not\subseteq\mathfrak P^c
(\mathcal X)\:.\]
\end{theorem}
\begin{proof}
Using Note~\ref{Narges2}, and Lemma~\ref{sara40}, $[g]\notin\mathfrak P^c(\mathcal X)$, thus
${\mathfrak P}^c(\mathcal X)\subset\pi_1(\mathcal X)$.
Using $[g]\in \mathfrak P_{\mathcal P_{fin}(\mathcal X)}^\omega
(\mathcal X)$ shows
$\mathfrak P_{\mathcal P_{fin}(\mathcal X)}^\omega
(\mathcal X)\not\subseteq{\mathfrak P}^c(\mathcal X)$. Also
$[g]\in\mathfrak P^\omega
(\mathcal X)\setminus \mathfrak P_{\mathcal P_{fin}(\mathcal X)}^\omega
(\mathcal X)$, thus $\mathfrak P^\omega
(\mathcal X)$ is a proper subgroup of $\mathfrak P_{\mathcal P_{fin}(\mathcal X)}^\omega
(\mathcal X)$. Using $[g]\in\mathfrak P_{\mathcal P_{fin}(\mathcal X)}^\omega
(\mathcal X)\setminus\mathfrak P^c
(\mathcal X)$ will complete the proof.
\end{proof}
\section{$\mathfrak
P_{{\mathcal P}_{fin}(\mathcal Y)}^c(\mathcal Y)$ is a proper subset of $\pi_1(\mathcal Y)$}
\noindent In this section we prove $\pi_1(\mathcal Y)\setminus \mathfrak
P_{{\mathcal P}_{fin}(\mathcal Y)}^c(\mathcal Y)\neq\varnothing$.
We use the same notations as in Section 7.
\\
Define $G:[0,1]\to\mathcal Y$ with:
\[G(x)=\left\{\begin{array}{lc}
    \dfrac{g(4xn(n+1)-(2n+1))}{2^{n+1}}+\dfrac1n &
        \dfrac{2n+1}{4n(n+1)}\leq x\leq\dfrac1{2n},n\in{\mathbb N} \: ,\\
        & \\
    2(n+1)(2n-1)x+(2-2n) & \dfrac1{2(n+1)}\leq x\leq\dfrac{2n+1}{4n(n+1)},n\in{\mathbb N}\:, \\
    & \\
    2-2x & \dfrac12\leq x\leq1 \:,\\
    &\\
    0 & x=0\:,
    \end{array}\right.\]
where $g:[0,1]\to{\mathcal X}$ as in section 7 is:
\[g(x)=\left\{\begin{array}{lc}
    \dfrac1ne^{2\pi i\frac{x-a_n}{b_n-a_n}}+\dfrac{i}{n}& x\in(a_n,b_n),n\in\mathbb N \:,\\
    &\\
    0 & {\rm otherwise}\:. \end{array}\right.\]
\begin{lemma}\label{bahar10}
Let $K,G:[0,1]\to\mathcal Y$ are homotopic and $m\in\mathbb N$, then
$|K^{-1}(\frac1m)|=c$.
\end{lemma}
\begin{proof}
Choose $\theta\in(\frac1{m+1}+\frac1{2^{m+2}},\frac1m-\frac1{2^{m+1}})$.
Consider $h,\overline h:[0,1]\to\mathcal Y$ with $h(x)=\theta x$
and $\overline h(x)=\theta(1-x)$. Since $K,G:[0,1]\to\mathcal Y$
are path homotopic with base point 0, $\overline h*K*h,\overline h*G*h:
[0,1]\to\mathcal Y$ are path homotopic with base point $\theta$.
Using Convention~\ref{good10} two maps
\[(\overline h*K*h)^{\{(x,y)\in\mathcal Y:x\geq\theta\}},
(\overline h*G*h)^{\{(x,y)\in\mathcal Y:x\geq\theta\}}:
[0,1]\to\{(x,y)\in\mathcal Y:x\geq\theta\}\]
are path homotopic
with base point $\theta$. Let
\[K_1=(\overline h*K*h)^{\{(x,y)\in\mathcal Y:x\geq\theta\}}
\:{\rm and}\:G_1=(\overline h*G*h)^{\{(x,y)\in\mathcal Y:x\geq\theta\}}\:.\]
If $m=1$ let $K_2=K_1$ and $G_2=G_1$.
\\
If $m>1$,
choose $\mu\in(\frac1m+\frac1{2^{m+1}},\frac1{m-1}-\frac1{2^m})$.
Consider $h_1,\overline h_1:[0,1]\to\mathcal Y$ with $h_1(x)=
(\mu-\theta)x+\theta$ and $\overline h_1(x)=(\mu-\theta)(1-x)+\theta$.
Since $K_1,G_1:[0,1]\to\{(x,y)\in\mathcal Y:x\geq\theta\}$
are path homotopic with base point $\theta$,
$\overline h_1*K_1*h_1,\overline h_1*G_1*h_1:
[0,1]\to\{(x,y)\in\mathcal Y:x\geq\theta\}$ are path homotopic with
base point $\mu$.
Using Convention~\ref{good10} two maps
$(\overline h_1*K_1*h_1)^{\{(x,y)\in\mathcal Y:\theta\leq x\leq\mu\}}$ and
$(\overline h_1*G_1*h_1)^{\{(x,y)\in\mathcal Y:\theta\leq x\leq\mu\}}$ from
$[0,1]$ to $\{(x,y)\in\mathcal Y:\theta\leq x\leq\mu\}$
are path homotopic
with base point $\mu$. Let $h_2(x)=(\frac1m-\mu)x+\mu$ and
$\overline h_2(x)=(\frac1m-\mu)(1-x)+\mu$
for $x\in[0,1]$.
\\
Now let:
\[K_2=\left\{\begin{array}{lc}
    \overline h_2*
    (\overline h_1*K_1*h_1)^{\{(x,y)\in\mathcal Y:\theta\leq x\leq\mu\}}*h_2
    & m>1 \: , \\
    K_1 & m=1 \: , \end{array}\right.\]
and
\[G_2=\left\{\begin{array}{lc}
    \overline h_2*
    (\overline h_1*G_1*h_1)^{\{(x,y)\in\mathcal Y:\theta\leq x\leq\mu\}}*h_2
    & m>1 \: , \\
    G_1 & m=1 \: , \end{array}\right.\]
also in order to be more convenient, whenever $m=1$ let $\mu=1$. Then
$K_2,G_2:[0,1]\to\{(x,y)\in\mathcal Y:\theta\leq x\leq\mu\}$ ($\subseteq
(\frac1{2^{m+1}}\mathcal X+\frac1m)\cup[\theta,\mu]$)
are path homotopic
with base point $\frac1m$. Hence there exists a continuous map
$F:[0,1]\times[0,1]\to\{(x,y)\in\mathcal Y:\theta\leq x\leq\mu\}$ such that
$F(0,s)=F(1,s)=\frac1m$, $F(s,0)=K_2(s)$ and $F(s,1)=G_2(s)$ for all $s\in[0,1]$.
\\
Define $\mathcal K , \mathcal G :[0,1]\to\mathcal X$ and
$\mathcal F:[0,1]\times[0,1]\to \mathcal X$ with:
\[\mathcal K(x)=\left\{\begin{array}{lc}
    2^{m+1}(K_2(x)-\frac1m) & K_2(x)\in\frac1{2^{m+1}}\mathcal X+\frac1m \: , \\
    -ie^{\frac{i\pi(K_2(x)-\frac1m)}2}+i & \theta\leq K_2(x)\leq \mu \: , 
    \end{array}\right.\]
\[\mathcal G(x)=\left\{\begin{array}{lc}
    2^{m+1}(G_2(x)-\frac1m) & G_2(x)\in\frac1{2^{m+1}}\mathcal X+\frac1m \: , \\
    -ie^{\frac{i\pi(G_2(x)-\frac1m)}2}+i & \theta\leq G_2(x)\leq \mu \: , 
    \end{array}\right.\]
\[\mathcal F(s,t)=\left\{\begin{array}{lc}
    2^{m+1}(F(s,t)-\frac1m) & F(s,t)\in\frac1{2^{m+1}}\mathcal X+\frac1m \: , \\
    -ie^{\frac{i\pi(F(s,t)-\frac1m)}2}+i & \theta\leq F(s,t)\leq \mu \: .
    \end{array}\right.\]
Using the gluing lemma, $\mathcal K$, $\mathcal G$ and $\mathcal F$
are continuous, moreover by the above definition, for all $s\in[0,1]$ we have:
\begin{itemize}
\item the equality $F(0,s)=F(1,s)=\frac1m$, implies
    \[\mathcal F(0,s)=\mathcal F(1,s)=-ie^{\frac{i\pi(\frac1m-\frac1m)}2}+i=0\:,\]
\item two equalities $F(s,0)=K_2(s)$ and $F(s,1)=G_2(s)$, imply
    $\mathcal F (s,0)=\mathcal K (s)$ and $\mathcal F (s,1)=\mathcal G (s)$.
\end{itemize}
So $\mathcal K, \mathcal G ;[0,1]\to\mathcal X$ are path homotopic with base point 0.
One could verify that $\mathcal G, g:[0,1]\to\mathcal X$ are homotopic, thus
 $\mathcal K, g:[0,1]\to\mathcal X$ are homotopic and by
Lemma~\ref{sara40} in which we proved $|f^{-1}(0)|=c$ whenever $f,g:[0,1]\to\mathcal X$
are homotopic, we have $|\mathcal K^{-1}(0)|\geq c$.
Since $\mathcal K^{-1}(0)$ and $K^{-1}(\frac1m)$ differs in a finite set,
we have $|K^{-1}(\frac1m)|= c$.
\end{proof}
\begin{theorem}\label{bahar20}
We have $[G]\in\pi_1(\mathcal Y)\setminus \mathfrak
P_{{\mathcal P}_{fin}(\mathcal Y)}^c(\mathcal Y)$.
\end{theorem}
\begin{proof}
If $[G]\in\mathfrak
P_{{\mathcal P}_{fin}(\mathcal Y)}^c(\mathcal Y)$, then by Note~\ref{Narges2}, there exists
$c\frac{{\mathcal P}_{fin}(\mathcal Y)}{}$loop $K:[0,1]\to\mathcal Y$ with
$K(0)=K(1)=0$ and $[F]=[G]$. Since $G:[0,1]\to\mathcal Y$ is not null-homotopic,
$K:[0,1]\to\mathcal Y$ is not constant. Thus there exists $k\in\mathbb N$ such that for all
$m\geq k$ we have $\frac1m\in K[0,1]$.  By Lemma~\ref{bahar10}, for all $m\geq k$ we have
$|K^{-1}(\frac1m)|=c$, thus $\{x\in{\mathcal Y}:|K^{-1}(x)|\not<c\}$ is infinite, which is
a contradiction, since $K:[0,1]\to\mathcal Y$ is a $c\frac{{\mathcal P}_{fin}(\mathcal Y)}{}$loop.
Therefore $[G]\notin\mathfrak
P_{{\mathcal P}_{fin}(\mathcal Y)}^c(\mathcal Y)$.
\end{proof}
\section{Main examples and counterexamples}
\noindent Now we are ready to present examples.
\begin{example}\label{example1}
Using Note~\ref{Narges4} (4), since $A(f_{\mathcal X})(=\mathbb N)$
is infinite, thus $[f_{\mathcal X}]\notin\mathfrak P^\omega(\mathcal X)$.
On the other hand, using Example~\ref{taha10} (1), $f_{\mathcal X}:[0,1]\to\mathcal X$ is a
$c-$loop, thus
$[f_{\mathcal X}]\in{\mathfrak P}^c(\mathcal X)\setminus\mathfrak P^\omega(\mathcal X)$
and $\mathfrak P^\omega(\mathcal X)$ is a proper subgroup of $\mathfrak P^c(\mathcal X)$.
Therefore by Theorem~\ref{taha10}, we have:
\[{\mathfrak P}^\omega(\mathcal X)\subset{\mathfrak P}^c(\mathcal X)\subset\pi_1(\mathcal X)\:.\]
Also using Theorem~\ref{taha10} again we have
$\mathfrak P^\omega
(\mathcal X)\subset\mathfrak P_{\mathcal P_{fin}(\mathcal X)}^\omega
(\mathcal X)$, which leads to
$\mathfrak P^\omega
(\mathcal X)\subset\mathfrak P_{\mathcal P_{fin}(\mathcal X)}^c
(\mathcal X)$, since
$\mathfrak P_{\mathcal P_{fin}(\mathcal X)}^\omega
(\mathcal X)\subseteq\mathfrak P_{\mathcal P_{fin}(\mathcal X)}^c
(\mathcal X)$. We recall that according to Theorem~\ref{taha10},
$\mathfrak P_{\mathcal P_{fin}(\mathcal X)}^\omega
(\mathcal X)\not\subseteq\mathfrak P^c
(\mathcal X)$, which leads to
$\mathfrak P_{\mathcal P_{fin}(\mathcal X)}^c
(\mathcal X)\not\subseteq\mathfrak P^c
(\mathcal X)$
since
$\mathfrak P_{\mathcal P_{fin}(\mathcal X)}^\omega
(\mathcal X)\subseteq\mathfrak P_{\mathcal P_{fin}(\mathcal X)}^c
(\mathcal X)$.
\end{example}
\noindent The following Example deal with Theorem~\ref{Narges3}.
We again recall that $\Phi:\pi_1(X)\times\pi_1(Y)\to\pi_1(X\times
Y)$, with $\Phi([f],[g])=[(f,g)]$ (where $(f,g)(t)=(f(t),g(t))$
(for $t\in[0,1]$ and loops $f:[0,1]\to X$, $g:[0,1]\to Y$)) is a
group isomorphism. Moreover as it was proved in
Theorem~\ref{Narges3} (4c), for infinite cardinal number $\alpha$
we have $\Phi({\mathfrak P}^\alpha(X)\times{\mathfrak
P}^\alpha(Y))
        \subseteq{\mathfrak P}^\alpha(X\times Y)$.
In the following we bring an example in which $\Phi({\mathfrak
P}^\alpha(X)\times{\mathfrak P}^\alpha(Y))
        \neq{\mathfrak P}^\alpha(X\times Y)$, in particular
we prove that
$\Phi\restriction_{{\mathfrak P}^\omega(\mathcal X)
\times{\mathfrak P}^\omega(\mathcal X)}:
{\mathfrak P}^\omega(\mathcal X)
\times{\mathfrak P}^\omega(\mathcal X)\to
{\mathfrak P}^\omega({\mathcal X}\times{\mathcal X})$
is a group monomorphism but it is not an isomorphism.
\begin{example}\label{example6}
Define $\overline f_{\mathcal X}:[0,1]\to {\mathcal X}$ with
$\overline f_{\mathcal X}(t)=f_{\mathcal X}(1-t)$.
$(f_{\mathcal X},\overline f_{\mathcal X}):[0,1]\to {\mathcal X}\times{\mathcal X}$
is an $\omega-$arc since for
all $(x,y)\in{\mathcal X}\times{\mathcal X}$, if
$|(f_{\mathcal X},\overline f_{\mathcal X})^{-1}(x,y)|>1$,
then $x=y=0$. Moreover $(f_{\mathcal X},\overline
f_{\mathcal X})^{-1}(0,0)\subseteq\{t\in[0,1]:t,1-t\in\{\frac1n:n\in{\mathbb
N}\}\}\cup\{0,1\}=\{0,1,\frac12\}$. Therefore for all $(x,y)$ we
have $|(f_{\mathcal X},\overline f_{\mathcal X})^{-1}(x,y)|\leq3<\omega$
and $(f_{\mathcal X},\overline
f_{\mathcal X}):[0,1]\to X\times X$ is an $\omega-$arc.
Thus $\Phi([f_{\mathcal X}],[\overline
f_{\mathcal X}])=[(f_{\mathcal X},\overline f_{\mathcal X})]\in
{\mathfrak P}^\omega({\mathcal X}\times{\mathcal X})$.
Since $\Phi:\pi_1({\mathcal X})\times\pi_1({\mathcal X})\to\pi_1
({\mathcal X}\times{\mathcal X})$ is a group isomorphism, there exist unique
$([g],[h])\in\pi_1({\mathcal X})\times\pi_1({\mathcal X})$ with
$\Phi([g],[h])=[(f_{\mathcal X},\overline f_{\mathcal X})]$
therefore $[g]=[f_{\mathcal X}]$ and $[h]=[\overline f_{\mathcal X}]$.
Using Example~\ref{example1},
$[f_{\mathcal X}]\notin{\mathfrak P}^\omega({\mathcal X})$,
so  $([g],[h])=
([f_{\mathcal X}],[\overline f_{\mathcal X}])\notin{\mathfrak
P}^\omega({\mathcal X})\times{\mathfrak P}^\omega({\mathcal X})$.
So (note: $\Phi$ is one to one):
\[[(f_{\mathcal X},\overline f_{\mathcal X})]=\Phi([g],[h])
=\Phi([f_{\mathcal X}],[\overline f_{\mathcal X}])
\notin\Phi({\mathfrak P}^\omega({\mathcal X})\times{\mathfrak P}^\omega({\mathcal X}))\]
which shows $\Phi({\mathfrak P}^\omega({\mathcal X})\times
{\mathfrak P}^\omega({\mathcal X}))\neq
{\mathfrak P}^\omega({\mathcal X}\times{\mathcal X})$.
\end{example}
\begin{example}\label{example3}
Using the same notations as in Note~\ref{Narges5} we have
$B(f_{\mathcal Y})=\mathbb N\times\mathbb N$, therefore for all
$m\in\mathbb N$, $\{n\in\mathbb N: (m,n)\in B(f_{\mathcal
Y})\}(=\mathbb N)$ is infinite. If $F\in\mathcal P_{fin}(\mathcal
Y)$, then $F$ is finite and there exists $k\in\mathbb N$ with
$\frac1k\in{\mathcal Y}\setminus F$. Using infiniteness of
$\{n\in\mathbb N: (k,n)\in B(f_{\mathcal Y})\}$ and
Note~\ref{Narges5} (c) we have $[f_{\mathcal Y}]\notin \mathfrak
P_{\mathcal P_{fin}(\mathcal Y)}^\omega(\mathcal Y)$. On the
other hand using Example~\ref{taha10} (2), $f_{\mathcal
Y}:[0,1]\to\mathcal Y$ is a $c-$loop, thus $[f_{\mathcal
Y}]\in{\mathfrak P}^c({\mathcal Y})$. So $\mathfrak P_{\mathcal
P_{fin}(\mathcal Y)}^\omega(\mathcal Y)$ is a proper subgroup of
${\mathfrak P}^c({\mathcal Y})$ and $\pi_1(\mathcal Y)$  (Hint:
We can prove $\mathfrak P^\omega(\mathcal Y)$ is a proper
subgroup of $\mathfrak P_{\mathcal P_{fin}(\mathcal
Y)}^\omega(\mathcal Y)$, thus $\mathfrak P^\omega(\mathcal
Y)\subset \mathfrak P_{\mathcal P_{fin}(\mathcal
Y)}^\omega(\mathcal Y)\subset\pi_1(\mathcal Y)$).
\end{example}
\begin{example}\label{example4}
Map $f_{\mathcal Z}:[0,1]\to\mathcal Z$ is a $p+1-$arc and it is
not homotopic with any $k-$arc $g:[0,1]\to{\mathcal Z}$ for
$k<p+1$. However for all $\alpha\geq2$ and ideal $\mathcal I$ on
${\mathcal Z}$ we have $\mathfrak P^\alpha_{\mathcal I}(\mathcal
Z)=\pi_1(\mathcal Z)$. For this aim, for all $k\in\{1,...,p\}$,
define $f_k:[0,1]\to{\mathcal Z}$ with $f_k(t)=\frac1ke^{2\pi
i(t-\frac14)}+\frac{i}{k}$. For all $\alpha\geq2$ and ideal
$\mathcal I$ on ${\mathcal Z}$, we have $[f_k]\in\mathfrak
P^2({\mathcal Z}) \subseteq\mathfrak P^\alpha_{\mathcal
I}({\mathcal Z}) \subseteq\pi_1({\mathcal Z})$. Since
$\{[f_1],...,[f_n]\}$ generates $\pi_1({\mathcal Z})$, thus
$\mathfrak P^\alpha_{\mathcal I}({\mathcal Z}) =\pi_1({\mathcal
Z})$.
\end{example}
\begin{example}\label{example8}
We recall that $\pi_1(\mathcal Y)\setminus \mathfrak P_{\mathcal
P_{fin}({\mathcal Y})}^c(\mathcal Y)\neq \varnothing$ by Theorem~\ref{bahar20}.
However $[f_{\mathcal Y}]\in \mathfrak
P_{\mathcal P_{fin}({\mathcal Y})}^c(\mathcal Y)$ (since $f_{\mathcal Y}:[0,1]\to\mathcal Y$
is a $c-$loop, thus $[f_{\mathcal Y}]\in\mathfrak
P_{\mathcal P_{fin}({\mathcal Y})}^c(\mathcal Y)$)
One may show $[f_{\mathcal Y}]\notin\mathfrak
P_{\mathcal P_{fin}({\mathcal Y})}^\omega(\mathcal Y)$, thus:
\[\mathfrak
P_{\mathcal P_{fin}({\mathcal Y})}^\omega(\mathcal Y) \subset \mathfrak
P_{\mathcal P_{fin}({\mathcal Y})}^c(\mathcal Y) \subset\pi_1(\mathcal Y)\:.\]
\end{example}
\section{Main Table}
\begin{jadval}\label{jadval}
We have the following Table:
\begin{center}
\begin{tabular}{l|c|c|c|c|c|}
    $\dfrac{\SP\SP\SP\SP\SP\SP K}{H\SP\SP\SP\SP\SP\SP}$ & $\mathfrak P^\omega(X)$ &
        $\mathfrak P_{\mathcal P_{fin}(X)}^\omega(X)$ & $\mathfrak P^c(X)$ &
    $\mathfrak P_{\mathcal P_{fin}(X)}^c(X)$ & $\pi_1(X)$ \\ \hline
    $\mathfrak P^\omega(X)$  & $\subseteq$ &
        $\subseteq$ & $\subseteq$ & $\subseteq$ & $\subseteq$ \\  \hline
    $\mathfrak P_{\mathcal P_{fin}(X)}^\omega(X)$ &
     \ref{example1} &
        $\subseteq$ & \ref{example1} & $\subseteq$ & $\subseteq$ \\  \hline
    $\mathfrak P^c(X)$ & \ref{example1} & \ref{example3} &
    $\subseteq$ & $\subseteq$ & $\subseteq$ \\ \hline
    $\mathfrak P_{\mathcal P_{fin}(X)}^c(X)$ & \ref{example1}
     & \ref{example3} & \ref{example1}
     & $\subseteq$ & $\subseteq$ \\ \hline
    $\pi_1(X)$ &
    \ref{example1} & \ref{example3} & \ref{example1} &
    \ref{example8} & $\subseteq$ \\   \hline
\end{tabular}  \\
$\:$ \\
In the above table ``$\subseteq$'' means that in the corresponding case
we have $H\subseteq K$.
\\
In addition the number {\it i.j} means that in Example~{\it i.j}
there exists an example such that $H\not\subseteq K$ in the corresponding case.
\end{center}
\end{jadval}
\section{Two spaces having fundamental groups isomorphic to Hawaiian earring's
fundamental group}
\noindent In this section we prove in a sequel of Lemmas, that
$\mathcal X$ (Hawaiian earring) and $\mathcal W$ are homeomorph
with two deformation retracts of $\mathcal V$. Thus we have
$\pi_1({\mathcal X})\cong\pi_1({\mathcal V})\cong\pi_1({\mathcal
W})$, which is important for our main counterexamples in next
section.
\\
We recall sign map ${\rm sgn}:{\mathbb R}\to\{\pm1,0\}$ with
${\rm sgn}(x)=\frac{x}{|x|}$ for $x\neq0$ and ${\rm sgn}(0)=0$.
\\
Note: In a connected topological space $A$, we call $x\in A$ a cut point of $A$
if $A\setminus\{x\}$ is disconnected.
It is evident that $\mathcal X$ and $\mathcal W$ are not homeomorphic since
$\mathcal X$ has just one cut point and $\mathcal W$ has infinitely many cut points.
\begin{lemma}\label{salam1}
For $x\in[0,1]$, the map
$\Phi_x:[0,\frac12]\to\{w\in[-1,1]:x+w\leq0\}=[-1,1]\cap(-\infty,-x]=[-1,-x]$ with:
\[\Phi_x(t)=\left\{\begin{array}{lc}
    (1-\sin(\pi t))(1-\frac{x}{1-2t})-1 & t\in[0,\frac12) \\
    -1 & t=\frac12 \end{array}\right.\]
is a homeomorphism.
\end{lemma}
\begin{proof}
Suppose $z\in(-1,1]$ and $z+x\leq0$.
The map $\varphi:[0,\frac12)\to {\mathbb R}$ with
$\varphi(t)=(1-\sin(\pi t))(1-\frac{x}{1-2t})-1$ is continuous, moreover
$\varphi(0)=-x$ and ${\displaystyle\lim_{\SP t\to{\frac12}^-}\varphi(t)}=-1$.
By $-1< z\leq -x$ and the mean value theorem there exists $t\in[0,\frac12)$ with
$\varphi(t)=z$. In addition
$\Phi_x\restriction_{[0,\frac12)}=\varphi:[0,\frac12)\to {\mathbb R}$
is strictly decreasing, therefore
$\Phi_x:[0,\frac12]\to[-1,-x]$
is a bijective continuous map
which completes the proof.
\end{proof}
\begin{lemma}\label{salam3}
Using the same notations as in Lemma~\ref{salam1},
$\widehat{\Phi}:\{(x,w)\in[0,1]\times[-1,0]:x+w\leq0\}\to[0,\frac12]$
with $\widehat{\Phi}(x,w)=\Phi_x^{-1}(w)$ is continuous.
\end{lemma}
\begin{proof}
Using Lemma~\ref{salam1},
$\widehat{\Phi}:\{(x,w)\in[0,1]\times[-1,0]:x+w\leq0\}\to[0,\frac12]$
is well-defined. Let $A:=\{(x,w)\in[0,1]\times[-1,0]:x+w\leq0\}$.
Consider $(x,w)\in A$, $s\in[0,\frac12]$, and sequence
$\{(x_n,w_n):n\in{\mathbb N}\}$ such that
${\displaystyle\lim_{n\to\infty}x_n}=x$,
${\displaystyle\lim_{n\to\infty}w_n}=w$,
${\displaystyle\lim_{n\to\infty}\widehat{\Phi}(x_n,w_n)}=s$.
Let $t=\widehat{\Phi}(x,w)$ and $t_n=\widehat{\Phi}(x_n,w_n)$ ($n\in\mathbb N$).
We show $s=t$, i.e.
${\displaystyle\lim_{n\to\infty}\widehat{\Phi}(x_n,w_n)}=\widehat{\Phi}(x,w)$.
\\
We have the following cases:
\\
{\it Case 1}.
$s\neq\frac12$.
In this case we may suppose for all $n\in\mathbb N$ we have
$t_n\neq\frac12$. For all $n\in\mathbb N$ we have
$w_n=\Phi_{x_n}(t_n)=(1-\sin(\pi t_n))(1-\frac{x_n}{1-2t_n})-1$
moreover:
\begin{eqnarray*}
\Phi_x(t) & = & w = {\displaystyle\lim_{n\to\infty}w_n}
        = {\displaystyle\lim_{n\to\infty}\Phi_{x_n}(t_n)} \\
    & = & {\displaystyle\lim_{n\to\infty}(1-\sin(\pi t_n))(1-\frac{x_n}{1-2t_n})-1} \\
    & = & (1-\sin(\pi s))(1-\frac{x}{1-2s})-1 =\Phi_x(s)
\end{eqnarray*}
and
$s=t$ since $\Phi_x$ is one to one according to Lemma~\ref{salam1}.
\\
{\it Case 2}. $s=\frac12$ and for infinitely many of $n$s we have $t_n=\frac12$.
In this case we may suppose for all $n\in\mathbb N$ we have $t_n=\frac12$.
Thus we have:
\begin{eqnarray*}
\Phi_x(t) & = & w = {\displaystyle\lim_{n\to\infty}w_n}
        = {\displaystyle\lim_{n\to\infty}\Phi_{x_n}(t_n)} \\
    & = & {\displaystyle\lim_{n\to\infty}\Phi_{x_n}(\frac12)}
        ={\displaystyle\lim_{n\to\infty}-1}=-1=\Phi_x(\frac12)
\end{eqnarray*}
and
$s=\frac12=t$ since $\Phi_x$ is one to one according to Lemma~\ref{salam1}.
\\
{\it Case 3}. $s=\frac12$ and for infinitely many of $n$s we have $t_n\neq\frac12$.
In this case we may suppose for all $n\in\mathbb N$ we have $t_n\neq\frac12$.
Thus we have:
\begin{eqnarray*}
\Phi_x(t) & = & w = {\displaystyle\lim_{n\to\infty}w_n}
        = {\displaystyle\lim_{n\to\infty}\Phi_{x_n}(t_n)} \\
    & = & {\displaystyle\lim_{n\to\infty}(1-\sin(\pi t_n))(1-\frac{x_n}{1-2t_n})-1} \\
    & = & {\displaystyle\lim_{n\to\infty}\frac{(1-\sin(\pi t_n))}{1-2t_n}
        \lim_{n\to\infty}(1-2t_n-x_n)}-1 \\
    & = & 0\times(1-s-x)-1=-1=\Phi_x(\frac12)
\end{eqnarray*}
and
$s=\frac12=t$ since $\Phi_x$ is one to one according to Lemma~\ref{salam1}.
\\
Using the above cases $s=t$ and
$\widehat{\Phi}:\{(x,w)\in[0,1]\times[-1,0]:x+w\leq0\}\to[0,\frac12]$
is continuous (otherwise since $[0,\frac12]$ is compact,
there exists $(x,w)\in A$ and sequence $\{(x_n,w_n):n\in\mathbb N\}$
converging to $(x,w)$ such that the sequence
$\{\widehat{\Phi}(x_n,w_n):n\in\mathbb N\}$ converges to a
point $s\in[0,\frac12]\setminus\{\widehat{\Phi}(x,w)\}$).
\end{proof}
\begin{lemma}\label{salam5}
Consider $X=\{(x,y,z)\in{\mathbb R}^3:y^2+z^2=1,0\leq x\leq1\}$
and $\widehat{\Phi}$
as in Lemma~\ref{salam3}. Let
$M_1=\{(x,y,z)\in X:x+z\leq0\}$,
the map
$F_1:[0,1]\times M_1\to X$ with $F_1(\mu,(x,y,z))=(x',y',z')$ for:
\[\left\{\begin{array}{l} x'=x+(1-2(1-x)\widehat{\Phi}(x,z)-x)\mu \: , \\
    z'=(1-\mu)z-\mu \: , \\
    y'={\rm sgn}(y)\sqrt{1-z'^2} \: , 
    \end{array}\right.\]
is continuous.
\end{lemma}
\begin{proof}
Let $(\mu,(x,y,z))\in[0,1]\times M_1$,
since $\widehat{\Phi}(x,z)\in[0,\frac12]$, we have
$1-2\widehat{\Phi}(x,z)\in[0,1]$ which leads to (use $x,\mu\in[0,1]$):
\begin{eqnarray*}
0\leq x(1-\mu)&=&x+(0-x)\mu\leq x'= x+(1-2(1-x)\widehat{\Phi}(x,z)-x)\mu \\
&\leq& x+(1-x)\mu\leq x+(1-x)=1
\end{eqnarray*}
thus $x'\in[0,1]$. Moreover using $\mu\in[0,1]$ and $z\in[-1,0]$ we have:
\[-1 =(1-\mu)(-1)-\mu\leq (1-\mu)z-\mu\leq(1-\mu)0-\mu=-\mu\leq0\: ,\]
thus $z'\in[-1,0]$ using $y'^2+z'^2=1$, $F_1:[0,1]\times M_1\to X$
is well-defined.
\\
Using Lemma~\ref{salam3}, $F_1:[0,1]\times M_1\to X$ is continuous.
\end{proof}
\begin{lemma}\label{salam7}
For $x\in[0,1]$, the map
$\Psi_x:[0,\frac12]\to\{z\in[-1,1]:x+z\geq0\}=[-1,1]\cap[-x,+\infty)=[-x,1]$
with $\Psi_x(t)=\sin(\pi t)-(1+\sin(\pi t)-4t)x$
is a homeomorphism.
\end{lemma}
\begin{proof}
Suppose $z\in[-1,1]$ and $z+x\geq0$.
Since
$\Psi_x(0)=-x$ and $\Psi_x(\frac12)=1$ by
the mean value theorem there exists $t\in[0,\frac12]$ with
$\Psi_x(t)=z$. Thus
$\Psi_x:[0,\frac12]\to[-x,1]$
is a bijection continuous map
which completes the proof.
\end{proof}
\begin{lemma}\label{salam9}
Using the same notations as in Lemma~\ref{salam7},
$\widehat{\Psi}:\{(x,w)\in[0,1]\times[-1,1]:x+w\geq0\}\to[0,\frac12]$
with $\widehat{\Psi}(x,w)=\Psi_x^{-1}(w)$ is continuous.
\end{lemma}
\begin{proof}
Using Lemma~\ref{salam7},
$\widehat{\Psi}:\{(x,w)\in[0,1]\times[-1,1]:x+w\geq0\}\to[0,\frac12]$
is well-defined. Let $B:=\{(x,w)\in[0,1]\times[-1,1]:x+w\geq0\}$.
Consider $(x,w)\in B$, $s\in[0,\frac12]$, and sequence
$\{(x_n,w_n):n\in{\mathbb N}\}$ such that
${\displaystyle\lim_{n\to\infty}x_n}=x$,
${\displaystyle\lim_{n\to\infty}w_n}=w$,
${\displaystyle\lim_{n\to\infty}\widehat{\Psi}(x_n,w_n)}=s$.
Let $t=\widehat{\Psi}(x,w)$ and $t_n=\widehat{\Psi}(x_n,w_n)$ ($n\in\mathbb N$).
We have:
\begin{eqnarray*}
\Psi_x(t) & = & w= {\displaystyle\lim_{n\to\infty}w_n}
    ={\displaystyle\lim_{n\to\infty}\Psi_{x_n}(t_n)} \\
& = & {\displaystyle\lim_{n\to\infty}\sin(\pi t_n)-(1+\sin(\pi t_n)-4t_n)x_n} \\
& = & \sin(\pi s)-(1+\sin(\pi s)-4s)x=\Psi_x(s)
\end{eqnarray*}
and
$s=t$ since $\Psi_x$ is one to one according to Lemma~\ref{salam7}.
Using the above discussion and the compactness of $[0,\frac12]$,
$\widehat{\Psi}:\{(x,w)\in[0,1]\times[-1,1]:x+w\geq0\}\to[0,\frac12]$
is continuous.
\end{proof}
\begin{lemma}\label{salam11}
Consider $X=\{(x,y,z)\in{\mathbb R}^3:y^2+z^2=1,0\leq x\leq1\}$
and $\widehat{\Psi}$
as in Lemma~\ref{salam9}. Let
$M_2=\{(x,y,z)\in X:x+z\geq0\}$,
the map
$F_2:[0,1]\times M_2\to X$ with $F_2(\mu,(x,y,z))=(x',y',z')$ for:
\[\left\{\begin{array}{l} x'=x+(1-x)\mu \: , \\
    z'=(1-\mu)z+(4\widehat{\Psi}(x,z)-1)\mu \: , \\
    y'={\rm sgn}(y)\sqrt{1-z'^2} \: ,
    \end{array}\right.\]
is continuous.
\end{lemma}
\begin{proof}
Let $(\mu,(x,y,z))\in[0,1]\times M_2$,
since $x,\mu\in[0,1]$ we have $0\leq x\leq x+(1-x)\mu\leq x+(1-x)=1$ and $x'\in[0,1]$.
Since $\widehat{\Psi}(x,z)\in[0,\frac12]$ we have $1-4\widehat{\Psi}(x,z)\in[-1,1]$.
Now using $\mu\in[0,1]$ and $1-4\widehat{\Psi}(x,z),z\in[-1,1]$ we have
\[-1=(1-\mu)(-1)+(-1)\mu\leq(1-\mu)z+(4\widehat{\Psi}(x,z)-1)\mu\leq1-\mu+\mu=1\]
therefore $z'\in[-1,1]$ using $y'^2+z'^2=1$, $F_2:[0,1]\times M_2\to X$
is well-defined.
\\
Using Lemma~\ref{salam9}, $F_2:[0,1]\times M_2\to X$ is continuous.
\end{proof}
\begin{construction}\label{construction1}
Consider $X=\{(x,y,z)\in{\mathbb R}^3:y^2+z^2=1,0\leq x\leq1\}$,
$Y=\{(x,y,z)\in X:x=1\vee z=-1\}$,
$M_1=\{(x,y,z)\in X:x+z\leq0\}$ and $M_2=\{(x,y,z)\in X:x+z\geq0\}$.
$F:[0,1]\times X\to X$ with $F\restriction_{M_1}=F_1$ as in Lemma~\ref{salam5}
and $F\restriction_{M_2}=F_2$ as in Lemma~\ref{salam11}. Then we have:
\begin{itemize}
\item[1.] $F:[0,1]\times X\to X$ is continuous.
\item[2.] For all $(x,y,z)\in X$ we have $F(0,(x,y,z))=(x,y,z)$ and $F(1,(x,y,z))\in Y$
\item[3.] For all $(x,y,z)\in Y$ and $\mu\in[0,1]$ we have $F(\mu,(x,y,z))=(x,y,z)$.
\end{itemize}
\end{construction}
\begin{proof}
(1) For all $x\in[0,1]$ we have $\widehat{\Phi}(x,-x)=\widehat{\Psi}(x,-x)=0$,
so using Lemma~\ref{salam5}, Lemma~\ref{salam11} and gluing lemma
the map $F:[0,1]\times X\to X$ is continuous.
\\
(2) For $(x,y,z)\in X$, $F(0,(x,y,z))=(x,y,z)$
is clear by definition of $F_1$ and $F_2$.
Suppose $F(1,(x,y,z))=(x_1,y_1,z_1)$. If $(x,y,z)\in M_1$,
then $z_1=(1-1)z-1=-1$ and $F(1,(x,y,z))=(x_1,y_1,z_1)\in Y$.
If $(x,y,z)\in M_2$, then $x_1=x+(1-x)1=1$ and
$F(1,(x,y,z))=(x_1,y_1,z_1)\in Y$.
\\
(3) Suppose $(x,y,z)\in Y$, $\mu\in[0,1]$ and $F(\mu,(x,y,z))=(x',y',z')$. We have
the following cases:
\\
{\it Case 1}. $z=-1$. In this case $y=0$, $(x,y,z)\in M_1$
and $\widehat{\Phi}(x,z)=\widehat{\Phi}(x,-1)=\frac12$.
Thus $x'=x+(1-2(1-x)\frac12-x)\mu=x$, $z'=(1-\mu)(-1)-\mu=-1=z$
and $y'={\rm sgn}(y)\sqrt{1-z'^2}={\rm sgn}(y)\sqrt{1-1}=0=y$
\\
{\it Case 2}. $x=1$. In this case $y=0$,
$(x,y,z)\in M_2$ and $\widehat{\Psi}(x,z)=\widehat{\Psi}(1,z)=t$ implies
$z=\Psi_1(t)=4t-1$, i.e. $\widehat{\Psi}(1,z)=\frac{z+1}4$. Thus
$x'=1+(1-1)\mu=1=x$, $z'=(1-\mu)z+(4\times\frac{z+1}4-1)\mu=z$,
and $y'={\rm sgn}(y)\sqrt{1-z'^2}={\rm sgn}(y)\sqrt{1-z^2}={\rm sgn}(y)|y|=y$.
\\
Considering the above cases we are done.
\end{proof}
\begin{construction}\label{construction2}
For $n\in\mathbb N$ let
\[X_n=\{(x,y,z)\in{\mathbb R}^3:y^2+(z-\frac1{n})^2=\frac1{n^2},0\leq x\leq\frac1n\}\:,\]
and $X_0=\bigcup\{X_n:n\in\mathbb N\}$, in this construction
we want to define a map $F_0:[0,1]\times X_0\to X_0$.
\\
Considering the same notations as in Construction~\ref{construction1} for $m\in\mathbb N$
and $(x,y,z)\in X_m$
we have $(mx,my,mz-1)\in X$. For $\mu\in[0,1]$ if
\[F(\mu,(mx,my,mz-1))=(x'_m,y'_m,z'_m)\in X\:,\]
then
$0\leq x'_m\leq1$ and $y'^2_m+z'^2_m=1$, thus $0\leq\frac{x'_m}{m}\leq\frac1m$ and
\[\left(\frac{y'_m}{m}\right)^2+\left(\frac{z'_m+1}{m}-\frac1{m}\right)^2=\frac1{m^2}\:,\]
therefore $(\frac{x'_m}{m},\frac{y'_m}{m},\frac{z'_m+1}{m})\in X_m$,
let
$F_m(\mu,(x,y,z))=(\frac{x'_m}{m},\frac{y'_m}{m},\frac{z'_m+1}{m})$,
i.e.
\[F_m(\mu,(x,y,z))=\frac1{m}F(\mu,(mx,my,mz-1))+(0,0,\frac1m)\:.\]
It is clear that $F_m:[0,1]\times X_m\to X_m$ is continuous. Suppose $s,t\in\mathbb N$,
$s<t$, $\mu\in[0,1]$ and
$(x,y,z)\in F_s\cap F_t$, then:
\[0\leq x\leq\min(\frac1t,\frac1s)\wedge y^2+(z-\frac1s)^2=\frac1{s^2}\wedge
y^2+(z-\frac1t)^2=\frac1{t^2}\]
which leads to $0\leq x\leq\frac1t(<\frac1s)$ and $y=z=0$.
Now, since $(sx,0,-1),(tx,0,-1)\in Y$ (in Construction~\ref{construction2}),
we have:
\[F(\mu,(sx,0,-1))=(sx,0,-1)\:,\:F(\mu,(tx,0,-1))=(sx,0,-1)\:,\]
and:
{\small \begin{eqnarray*}
F_s(\mu,(x,y,z))& = & F_s(\mu,(x,0,0))=\frac1{s}F(\mu,(sx,0,-1))+(0,0,\frac1s) \\
& = & \frac1s(sx,0,-1)+(0,0,\frac1s)=(x,0,0)=\frac1s(tx,0,-1)+(0,0,\frac1t) \\
& = & \frac1{t}F(\mu,(tx,0,-1))+(0,0,\frac1t)
    =F_t(\mu,(x,0,0))=F_t(\mu,(x,y,z))
\end{eqnarray*}}
Therefore for $F_0=\bigcup\{F_n:n\in\mathbb N\}$,
$F_0:[0,1]\times X_0\to X_0$ is well-defined.
\end{construction}
\noindent Note: We recall that for $A\subseteq B$, we call $A$ a
deformation retract of $B$ if there exists a continuous map
$\nu:[0,1]\times B\to A$ with $\nu(0,b)=b$, $\nu(1,b)\in A$, and
$\nu(t, a)=a$ (for all $b\in B,a\in A,t\in[0,1]$). It is
well-known that if $A$ is a deformation retract of $B$ (and
$a_0\in A$), then
$\mathop{\Upsilon:\pi_1(A,a_0)\to\pi_1(B,a_0)}\limits_{\SP\SP[k]\mapsto[k]}$
is a group isomorphism, in particular $\pi_1(A)\cong\pi_1(B)$
\cite[Theorem 58.3]{M07}.
\begin{lemma}\label{salam13}
For $n\in\mathbb N$ let
\[X_n=\{(x,y,z)\in{\mathbb R}^3:y^2+(z-\frac1{n})^2=\frac1{n^2},0\leq x\leq\frac1n\}\:,\]
and
\[Y_n=\{(x,y,z)\in X_n:x=\frac1{n}\vee z=0\}\:,\]
then $Y_0=\bigcup\{Y_n:n\in\mathbb N\}$
is a deformation retract of $X_0=\bigcup\{X_n:n\in\mathbb N\}$.
\end{lemma}
\begin{proof}
Consider
$F_0:[0,1]\times X_0\to X_0$ as in Construction~\ref{construction2}.
We prove the following claims:
\begin{itemize}
\item Claim 1. $F_0:[0,1]\times X_0\to X_0$ is continuous.
\item Claim 2. $\forall(x,y,z)\in X_0\SP(F_0(0,(x,y,z))=(x,y,z)\wedge F_0(1,(x,y,z))\in Y_0)$.
\item Claim 3. $\forall(x,y,z)\in Y_0\:\forall\mu\in[0,1]\SP F_0(\mu,(x,y,z))=(x,y,z)$.
\end{itemize}
{\it Proof of Claim 1}. Since for all $n\in\mathbb N$, $F_n:[0,1]\times X_n\to X_n$
is continuous, using the gluing lemma,
$\bigcup\{F_i:1\leq i\leq n\}:[0,1]\times\bigcup\{X_i:1\leq i\leq n\}
\to\bigcup\{X_i:1\leq i\leq n\}$ is continuous.
\\
If $(x,y,z)\in X_0\setminus\{(0,0,0)\}$, then there exist
$n\in\mathbb N$ and open neighborhood $V$ of $(x,y,z)$ in $X_0$
such that $V\subseteq\bigcup\{X_i:1\leq i\leq n\}$. Since
$\bigcup\{F_i:1\leq i\leq n\}\restriction_{[0,1]\times V}
:[0,1]\times\bigcup\{X_i:1\leq i\leq n\} \to\bigcup\{X_i:1\leq
i\leq n\}$ is continuous, $\bigcup\{F_i:1\leq i\leq
n\}:[0,1]\times V \to X_0$ is continuous, i.e.
$F_0\restriction_{[0,1]\times V}:[0,1]\times V \to X_0$ is
continuous, therefore $F_0$ is continuous in all points of
$[0,1]\times\{(x,y,z)\}$.
\\
In order to show the continuity of $F_0:[0,1]\times X_0\to X_0$,
we should prove that it is continuous in all points
$(\mu,(0,0,0))$ ($\mu\in[0,1]$). Consider $\varepsilon>0$ there exists $n\in\mathbb N$
such that $\frac{\sqrt6}n<\varepsilon$ for all $(x,y,z)\in X_0$ and $\mu,\lambda\in[0,1]$
we have (consider $[0,1]\times X_0$ and $X_0$ respectively
under Euclidean norm of ${\mathbb R}^4$ and ${\mathbb R}^3$):

$||(\mu,(0,0,0))-(\lambda,(x,y,z))||<\frac1n $
\begin{eqnarray*}
& \Rightarrow & x<\frac1n \\
& \Rightarrow & (x,y,z)\in\bigcup\{X_i:i\geq n\} \\
& \Rightarrow & F_0(\lambda,(x,y,z))=\bigcup\{F_i:i\geq n\}(\lambda,(x,y,z)) \\
& \Rightarrow & F_0(\lambda,(x,y,z))\in
    \bigcup\{F_i:i\geq n\}([0,1]\times\bigcup\{X_i:1\leq i\geq n\}) \\
& \Rightarrow & F_0(\lambda,(x,y,z))\in
    \bigcup\{F_i([0,1]\times X_i):i\geq n\}=\bigcup\{X_i:i\geq n\} \\
& \Rightarrow & ||F_0(\lambda,(x,y,z))||
    \leq\max\{\frac{\sqrt6}{i}:i\geq n\}=\frac{\sqrt6}{n} \\
& \Rightarrow & ||F_0(\lambda,(x,y,z))-F_0(\mu,(0,0,0))||=||F_0(\lambda,(x,y,z))||\leq
    \frac{\sqrt6}{n}<\varepsilon
\end{eqnarray*}
(note to the fact that $X_n\subseteq[0,\frac1n]\times[-\frac1n,\frac1n]\times[0,\frac2n]$,
thus for all $(u,v,w)\in X_n$ we have
$||(u,v,w)||\leq\sqrt{\frac1{n^2}+\frac1{n^2}+\frac4{n^2}}=\frac{\sqrt6}{n}$) therefore
$F_0:[0,1]\times X_0\to X_0$ is continuous in $(\mu,(0,0,0))$ as well as
it is continuous in other points of $[0,1]\times X_0$.
\\
{\it Proof of Claim 2}. Suppose $(x,y,z)\in X_0$, there exists $n\in\mathbb N$
such that $(x,y,z)\in X_n$, using Construction~\ref{construction1} (2), we have:
\begin{eqnarray*}
F_0(0,(x,y,z)) & = & F_n(0,(x,y,z))=\frac1{n}F(0,(nx,ny,nz-1))+(0,0,\frac1n) \\
& = & \frac1{n}(nx,ny,nz-1)+(0,0,\frac1n)=(x,y,z)
\end{eqnarray*}
and
$F_0(1,(x,y,z))=F_n(1,(x,y,z))=\frac1{n}F(1,(nx,ny,nz-1))+(0,0,\frac1n)$,
by Construction~\ref{construction1} (2)
we have $F(1,(nx,ny,nz-1))\in Y$
which leads to
$F_0(1,(x,y,z))\in\frac1{n}Y+(0,0,\frac1n)=Y_n\subseteq Y_0$.
\\
{\it Proof of Claim 3}. Suppose $\mu\in[0,1]$ and
$(x,y,z)\in X_0$, there exists $n\in\mathbb N$
such that $(x,y,z)\in Y_n\subseteq X_n$, now we have
(use Construction~\ref{construction1} (3)):
\begin{eqnarray*}
(x,y,z)\in Y_n & \Rightarrow &
    ((x,y,z)\in X_n\wedge x=\frac1n)\vee ((x,y,z)\in X_n\wedge z=0) \\
& \Rightarrow & (y^2+(z-\frac1n)^2=\frac1{n^2}\wedge x=\frac1n)
    \vee(0\leq x\leq\frac1n\wedge y=z=0) \\
& \Rightarrow & (((ny)^2+(nz-1)^2=1\wedge nx=1) \\
    & & \SP\SP\SP\SP\vee(0\leq nx\leq1\wedge ny=0\wedge nz-1=-1)) \\
& \Rightarrow & (nx,ny,nz-1)\in Y \\
& \Rightarrow & F(\mu,(nx,ny,nz-1))=(nx,ny,nz-1)
\end{eqnarray*}
thus
\begin{eqnarray*}
F_0(\mu,(x,y,z)) & = & F_n(\mu,(x,y,z))=\frac1{n}F(\mu,(nx,ny,nz-1))+(0,0,\frac1n) \\
    & = & \frac1{n}(nx,ny,nz-1)+(0,0,\frac1n)=(x,y,z)
\end{eqnarray*}
Which completes the proof of Claim 3.
\\
Using Claims 1, 2, and 3, $Y_0$ is a deformation retract of $X_0$.
\end{proof}
\begin{theorem}\label{retract}
Under the same notations as in Construction~\ref{construction2} and 
\linebreak
Lemma~\ref{salam13},
$Z_0=\{(0,y,z):\exists x\:(x,y,z)\in X_0\}$ is a deformation retract
of $X_0$. In particular 
$\pi_1(Y_0)\cong\pi_1(X_0)\cong\pi_1(Z_0)$.
\end{theorem}
\begin{proof}
The map $\mathop{[0,1]\times X_0\to Z_0}\limits{(\mu,(x,y,z))\mapsto
((1-\mu)x,y,z)}$ shows that $Z_0$ is a deformation retract of $X_0$ too. Now use
\cite[Theorem 58.3]{M07} to complete the proof.
\end{proof}
\begin{corollary}\label{retract10}
Two sets
    $\mathcal X$ and $\mathcal W$ are homeomorphic with deformation retracts of $\mathcal V$,
    therefore $\pi_1(\mathcal X)\cong\pi_1(\mathcal V)\cong\pi_1(\mathcal W)$.
\end{corollary}
\begin{proof}
Under the same notations as in Theorem~\ref{retract},
    $\mathcal X$ and $Z_0$ are homeomorph, moreover $Y_0$ and $\mathcal W$
    are homeomorph too, also $X_0=\mathcal V$. Now by Theorem~\ref{retract} we have
    $\pi_1({\mathcal X})\cong\pi_1({\mathcal V})\cong\pi_1({\mathcal W})$.
\end{proof}
\section{A distinguished counterexample}
\noindent In Section 11 we have proved $\pi_1(\mathcal X)\cong\pi_1(\mathcal W)$, in this section
we prove $\mathfrak P^\omega(\mathcal X)\ncong\mathfrak
P^\omega(\mathcal W)$.
\begin{lemma}\label{tir10}
We have $|{\mathfrak P}^\omega(\mathcal X)|=\omega$.
\end{lemma}
\begin{proof}
For $n\in\mathbb N$ consider
    $\rho_n:[0,1]\to C_n$ with $\rho_n(t)=\frac1ne^{2\pi it-\frac{\pi i}2}+\frac{i}{n}$,
    then $\omega-$loops $\rho_n,\rho_m:[0,1]\to{\mathcal X}$ are homotopic if and only if $n=m$.
    Therefore $\{[\rho_n]:n\in\mathbb N\}$ is an infinite subset of ${\mathfrak P}^\omega(\mathcal X)$
    which leads to $|{\mathfrak P}^\omega(\mathcal X)|\geq\omega$.
    On the other hand as it has been mentioned in Note~\ref{Narges4} (4), if
    $[f]\in{\mathfrak P}^\omega(\mathcal X)$, then $|A(f)|<\omega$,
    which leads to
    ${\mathfrak P}^\omega(\mathcal X)\subseteq*\{\pi_1(C_n):n\in\mathbb N\}$, thus
    \begin{eqnarray*}
    |{\mathfrak P}^\omega(\mathcal X)| & \leq & |*\{\pi_1(C_n):n\in\mathbb N\}| \\
    & = & |\{\rho_{i_1}^{j_1}*\rho_{i_2}^{j_2}*\cdots*\rho_{i_m}^{j_m}:m\in\mathbb N,
        i_1,j_1,i_2,j_2,\ldots,i_m,j_m\in\mathbb Z\}| \\
    & \leq & |{\displaystyle\bigcup_{m\in\mathbb N}\{(i_1,j_1,\cdots,i_m,j_m):
        i_1,j_1,\ldots,i_m,j_m\in\mathbb Z\}}|=\omega
    \end{eqnarray*}
Hence
    $|{\mathfrak P}^\omega(\mathcal X)|=\omega$.
\end{proof}
\begin{lemma}\label{tir20}
We have $|{\mathfrak P}^\omega(\mathcal W)|=c$.
\end{lemma}
\begin{proof}
It is well-known that for all Hausdorff separable space $A$, $|C(A,\mathbb R^2)|\leq c$
where $C(A,B)$ denotes the collection of all continuous maps $\phi:A\to B$.
Therefore
\[|{\mathfrak P}^\omega(\mathcal W)|\leq
|C([0,1],\mathcal W)|\leq|C([0,1],\mathbb R^2)|=c\:.\]
Now for all $a=(a_n:n\in\mathbb N)\in\{0,1\}^{\mathbb N}$ define $f_a:[0,1]\to\mathcal W$ with:
\vspace{5mm}
\begin{center}\scalebox{0.8}{
   $f_a(x)=\left\{\begin{array}{lc}
        \dfrac{1}{2^{n+1}}e^{2\pi i(4n(n+1)x-(2n+1)+\frac34)}
        +\dfrac{i}{2^{n+1}}+\dfrac1n

        & \dfrac{2n+1}{4n(n+1)}\leq x\leq\dfrac1{2n},a_n=1,n\in{\mathbb N} \: , \\
   & \\
    4x-\dfrac1{n+1} &
        \dfrac{1}{2(n+1)}\leq x\leq\dfrac{2n+1}{4n(n+1)},a_n=1,n\in{\mathbb N} \: , \\
    & \\
    2x & \dfrac{1}{2(n+1)}\leq x\leq\dfrac1{2n},a_n=0,n\in{\mathbb N} \: , \\
    & \\
    0 & x=0 \: , \\
    & \\
    2-2x & \dfrac12\leq x\leq1 \: , 
    \end{array}\right.$}
    \end{center}
then $f_a:[0,1]\to\mathcal W$  is an $\omega-$loop, thus
$[f_a]\in{\mathfrak P}^\omega(\mathcal W)$. We claim that
$\psi:\{0,1\}^{\mathbb N}\to{\mathfrak P}^\omega(\mathcal X)$
with $\psi(a)=[f_a]$ ($a\in\{0,1\}^{\mathbb N}$) is one to one.
Let $a=(a_n:n\in\mathbb N),b=(b_n:n\in\mathbb N)\in\{0,1\}^{\mathbb N}$ and $a\neq b$,
then there exists $m\in\mathbb N$ such that $a_m\neq b_m$. Suppose $a_m=0$ and $b_m=1$.
Let $W:=\{\frac{1}{2^{m+1}}e^{2\pi i\theta}+\frac{1}{m}+\frac{i}{2^{m+1}}:
\theta\in[0,1]\}$. Since $f_a^W$ is constant map $\frac1m$, $[f_a^W]$ is
null-homotopic. However $[f_b^W]$ is not null-homotopic, thus
$[f_a^W]\neq[f_b^W]$ which leads to $[f_a]\neq[f_b]$ according to Convention~\ref{good10}.
Hence $\psi:\{0,1\}^{\mathbb N}\to{\mathfrak P}^\omega(\mathcal X)$ is one to one which
leads to $|{\mathfrak P}^\omega(\mathcal X)|\geq|\{0,1\}^{\mathbb N}|=c$ and
completes the proof.
\end{proof}
\begin{counterexample}[A Distinguished Counterexample]\label{example5}
Two groups $\pi_1(\mathcal X)$ and $\pi_1(\mathcal W)$ are isomorphic and two groups
$\mathfrak P^\omega(\mathcal X)$ and $\mathfrak P^\omega(\mathcal W)$ are non-isomorphic.
Briefly $\pi_1(\mathcal X)\cong\pi_1(\mathcal W)$ and
$\mathfrak P^\omega(\mathcal X)\ncong\mathfrak P^\omega(\mathcal W)$
(use Lemma~\ref{tir10}, Lemma~\ref{tir20}, and Corollary~\ref{retract10}).
\end{counterexample}
\section{A diagram and a hint}
\noindent Consider the following diagram:
\begin{center}
\scalebox{0.85}{
\begin{tabular}{c}$\xymatrix{
    \forall\alpha\geq\omega\:{\mathfrak P}^\alpha(X)
    \cong{\mathfrak P}^\alpha(Y)\ar@{->}[r]^{(I)}  &
    \forall\alpha\geq c\:{\mathfrak P}^\alpha(X)
    \cong{\mathfrak P}^\alpha(Y)\ar@{->}[r]^{(II)} &
    \forall\alpha\geq2^c\:{\mathfrak P}^\alpha(X)
    \cong{\mathfrak P}^\alpha(Y)\ar@{<->}[d]^{(III)} \\
\pi_1(X)\cong\pi_1(Y)\ar@{<->}[rr]_{(IV)}\ar@{->}[u]_{(V)}|{\nmid} &&
    \exists\alpha\geq 2^c\:{\mathfrak P}^\alpha(X)
    \cong{\mathfrak P}^\alpha(Y)}$\end{tabular}}
\end{center}
Arrows (III) and (IV) are valid regarding Theorem~\ref{Narges3}
(1). However by Counterexample~\ref{example5}, there exist $X, Y$
such that $\pi_1(X)\cong\pi_1(Y)$ and ${\mathfrak
P}^\omega(X)\cong{\mathfrak P}^\omega(Y)$, thus:
\[\pi_1(X)\cong\pi_1(Y)\wedge\neg(\forall\alpha\geq\omega\:{\mathfrak P}^\alpha(X)
    \cong{\mathfrak P}^\alpha(Y))\]
Hence the above diagram is valid. We have the following arising
problems:
\begin{problem}
Find a counterexample for arrow (I), i.e.
find $X$, $Y$ such that $\pi_1(X)\cong\pi_1(Y)$, ${\mathfrak P}^c(X)\cong{\mathfrak P}^c(Y)$
and ${\mathfrak P}^\omega(X)\ncong{\mathfrak P}^\omega(Y)$ (Hint:
is it true that
${\mathfrak P}^c(\mathcal X )\cong{\mathfrak P}^c(\mathcal W )$).
\end{problem}
\begin{problem}
Find a counterexample for arrow (II), i.e.
find $X$, $Y$ such that $\pi(X)_1\cong\pi_1(Y)$ and ${\mathfrak P}^c(X)\ncong{\mathfrak P}^c(Y)$.
\end{problem}
\section{A Strategy for Future and Conjecture}
\noindent Let's extend of the idea of this text to homotopy group of order $n$.
Let $b\in \mathbb S^n$ be a fixed point. For infinite cardinal
number $\alpha$ and ideal $\mathcal I$ on $X$ which contains all
finite subsets of $X$, if $f,g:\mathbb S^n\to X$ are
$\alpha\frac{\mathcal I}{}$maps, with $f(b)=g(b)$, then it is easy
to see that $f\vee
g:\mathbb S^n\to X$ is $\alpha\frac{\mathcal I}{}$map too.
So we may have the following definition.
\begin{definition}
For $a\in X$, by $\mathfrak P_{(n,\mathcal I)}^\alpha(X,a)$ we mean subgroup of
$\pi_n(X,a)$ generated by $\alpha\frac{\mathcal I}{}$maps with base point
$a$.
\end{definition}
\noindent It's evident by the definition that for
ideals ${\mathcal I}$, ${\mathcal J}$ on $X$
containing finite subsets,
transfinite cardinal number $\alpha$, and
$a\in X$, we have:
\begin{itemize}
\item If ${\mathcal I}\subseteq{\mathcal J}$, then
    $\mathfrak P_{(n,\mathcal I)}^\alpha(X,a)\subseteq\mathfrak
        P_{(n,\mathcal J)}^\alpha(X,a)$;
\item $\mathfrak P_{(n,\mathcal I\cap \mathcal J)}^\alpha(X,a)
    \subseteq \mathfrak P_{(n,\mathcal I)}^\alpha(X,a)
    \cap\mathfrak P_{(n,\mathcal    J)}^\alpha(X,a)$.
\end{itemize}
Now we are ready to the following conjecture:
\\
{\bf Conjecture.} Arc connected spaces $X$ and $Y$ are homeomorph if and only if
there exists a bijection $f:X\to Y$ such that for all nonzero cardinal number
$\alpha$ and all ideal $\mathcal I$ on $X$, $\mathfrak P_{\mathcal I}^\alpha(X)
\cong\mathfrak P_{f({\mathcal I})}^\alpha(Y)$.
\\
{\bf One more idea for future study.} Let's recall that in topological space $Z$ and $a,b\in Z$ for nonzero cardinal number $\beta$,
a collection $\Gamma$ of maps $f:[0,1]\to Z$ with $f(0)=a$ and $f(1)=b$, is called a $\beta-$separated
family of maps between $a$ and $b$ if for all distinct $g,h\in \Gamma$ we have
$|(g[0,1]\cap h[0,1])\setminus\{a,b\}|<\beta$~\cite[Definition~2.5]{AH16}. 
\\
Now for cardinal numbers
$\alpha,\beta>0$ and ideal $\mathcal I$ on $X$ we may
consider the collection $S(\mathcal{I},\alpha,\beta)$ consisting of all families $\Gamma$
such that $\Gamma$ is a collection of $\alpha\frac{\mathcal I}{\:}$loops with base point $a$
and a $\beta-$separated family of maps between $a$ and $a$. Suppose
\[\mathfrak{S}(\mathcal{I},\alpha,\beta):=\{<\{[f]:f\in\Gamma\}>:\Gamma\in S(\mathcal{I},\alpha,\beta)\}\:,\]
then $\mathfrak{S}(\mathcal{I},\alpha,\beta)$ is a ``poset'' under $\subseteq$ and a collection of subgroups of $\pi_1(X)$. For arc connected
spaces $X$ and $Y$ one may compare these ``type'' of collections of their fundamental groups to discover ``differents'' beween $X$ and $Y$.
\section{Conclusion}
\noindent In this paper, for arc connected locally compact
Hausdorff topological space $X$ (with at least two elements),
$a\in X$, nonzero cardinal number $\alpha$, and ideal $\mathcal
I$ on $X$ we introduce ${\mathfrak P}_{\mathcal I}^\alpha(X,a)$
as a subgroup of $\pi_1(X,a)$. We prove that for transfinite
$\alpha$ and $a,b\in X$ two groups ${\mathfrak P}_{\mathcal
I}^\alpha(X,a)$ and ${\mathfrak P}_{\mathcal I}^\alpha(X,b)$ are
isomorphic, therefore for transfinite $\alpha$ we denote
${\mathfrak P}_{\mathcal I}^\alpha(X,a)$ simply by ${\mathfrak
P}_{\mathcal I}^\alpha(X)$ and ${\mathfrak
P}_{\{\varnothing\}}^\alpha(X)$ simply by ${\mathfrak
P}^\alpha(X)$. Moreover for $\alpha\geq2^c$ we have ${\mathfrak
P}_{\mathcal I}^\alpha(X)=\pi_1(X)$, hence the most interest is
in $\omega\leq\alpha<2^c$ using GCH we prefer to study
$\alpha\in\{\omega,c\}$. We obtain that for Hawaiian earring
(infinite earring) $\mathcal X$, three groups ${\mathfrak
P}^\omega(\mathcal X)$, ${\mathfrak P}^c(\mathcal X)$, and
${\mathfrak P}^{2^c}(\mathcal X)(=\pi_1(\mathcal X))$ are
pairwise distinct. Also we introduce $\mathcal Y$ such that
${\mathfrak P}_{{\mathcal P}_{fin}({\mathcal Y})}^\omega(\mathcal
Y)$, ${\mathfrak P}_{{\mathcal P}_{fin}({\mathcal Y})}^c(\mathcal
Y)$, and ${\mathfrak P}_{{\mathcal P}_{fin}({\mathcal
Y})}^{2^c}(\mathcal Y) (=\pi_1(\mathcal Y))$ are pairwise
distinct. We find $\mathcal W$ such that $\pi_1(\mathcal
X)\cong\pi_1(\mathcal W)$ and ${\mathfrak P}^\omega(\mathcal
X)\ncong{\mathfrak P}^\omega(\mathcal W)$, this example leads us
to the fact that we can classify spaces with isomorphic first
homotopy groups using the concept of ${\mathfrak P}^\alpha(-)$s
(\textit{first homotopy groups with respect to
$\alpha\geq\omega$}). However investigating the structure of our
examples and specially Section 12, shows remarkable role of the
number of (locally) cut points their and order in $\alpha-$arcs,
$\alpha\frac{\mathcal I}{}$arcs, and our constructed subgroups of
first fundamental group.
\section*{Acknowledgement}
\noindent With special thanks to our friends
A. Hosseini, P. Mirzaei, and M. Nayeri for their helps and comments.

\vspace{5mm}
\[\underline{\SP\SP\SP\SP\SP\SP\SP\SP\SP\SP\SP\SP\SP\SP\SP\SP}\]
{\small
{\bf Fatemah Ayatollah Zadeh Shirazi},
Faculty of Mathematics, Statistics and Computer Science,
College of Science, University of Tehran ,
Enghelab Ave., Tehran, Iran
\\
({\it e-mail}: fatemah@khayam.ut.ac.ir)
\\
{\bf Fatemeh Ebrahimifar},
Faculty of Mathematics, Statistics and Computer Science,
College of Science, University of Tehran ,
Enghelab Ave., Tehran, Iran
\\
({\it e-mail}: ebrahimifar64@ut.ac.ir)
\\
{\bf Mohammad Ali Mahmoodi},
Faculty of Mathematics, Statistics and Computer Science,
College of Science, University of Tehran ,
Enghelab Ave., Tehran, Iran
\\
({\it e-mail}: m.a.mahmoodi@ut.ac.ir)}
\end{document}